\documentclass{amsart}
\usepackage{color}

\newtheorem{theorem}{Theorem}
\newtheorem{lemma}{Lemma}
\newtheorem{corollary}{Corollary}
\newtheorem{proposition}{Proposition}
\theoremstyle{conjecture}

\theoremstyle{definition}
\newtheorem{definition}{Definition}
\theoremstyle{question}
\newtheorem{question}{Question}
\theoremstyle{questions}

\theoremstyle{remark}
\newtheorem{remark}{Remark}
\theoremstyle{remarks}

\theoremstyle{example}
\newtheorem{example}{Example}
\numberwithin{equation}{section}

\begin{document}
\title{Homeomorphisms of $S^1$ and Factorization}

\author{Mark Dalthorp}
\address{910 NW Jones Avenue, Albany, OR 97321}
\email{markdalthorp@email.arizona.edu}

\author{Doug Pickrell}
\address{Mathematics Department, University of Arizona, Tucson, AZ 85721}
\email{pickrell@math.arizona.edu}

\begin{abstract} For each $n>0$ there is a one complex parameter family of
homeomorphisms of the circle consisting of linear fractional transformations
`conjugated by $z \to z^n$'. We show that these families are free of relations, which determines
the structure of `the group of homeomorphisms of finite type'. We next
consider factorization for more robust groups of homeomorphisms.
We refer to this as root subgroup factorization (because the factors correspond to root subgroups). We are
especially interested in how root subgroup factorization is related to triangular factorization (i.e. conformal welding), and correspondences between smoothness properties of the homeomorphisms and decay properties of the root subgroup parameters.
This leads to interesting comparisons with Fourier series and the theory of Verblunsky coefficients.

\end{abstract}

\maketitle

\setcounter{section}{-1}

\section{Introduction}

In this paper we consider the question of whether it is possible to factor an orientation preserving homeomorphism of the circle, belonging to a given group, as a composition of `linear fractional transformations conjugated by $z \to z^n$'. What we mean by factorization depends on the group of homeomorphisms we are considering. In the introduction we will start with the simplest classes of homeomorphisms and build up. For algebraic homeomorphisms, factorization is to be understood in terms of generators and relations.
For less regular homeomorphisms factorization involves limits and ordering, and in particular is highly asymmetric with respect to inversion.

\subsection{Diffeomorphisms of Finite Type}

Given a positive integer $n$ and $w_n \in\Delta:=\{w\in \mathbb C : \vert w\vert < 1\}$, define a function
$\phi_n : S^1 \to S^1$ by
\begin{equation}
\phi_n(w_n; z) := z\frac{(1 + \bar{ w}_nz^{-n})^{1/n}}
{(1 + w_nz^n)^{1/n}},\quad  \vert z\vert=1\end{equation}
It is straightforward to check that $\phi_n \in Diff(S^1)$, the group of orientation preserving
diffeomorphisms of $S^1$, and $\phi_n^{-1}(z)= \phi_n(-w_n; z)$. If $n$ is fixed, the group generated by
the $\phi_n$, as $w_n$ varies, is isomorphic to the $n$-fold covering of $PSU(1,1)$. We will refer to the subgroup of $Diff(S^1)$ generated by the $\phi_n$, as $w_n$ and $n$ vary, as the group of diffeomorphisms of finite type (which we denote by $FTypeDiff$).

\begin{theorem}\label{theorem1} (a) If $m$ and $n$ are relatively prime, then the set of diffeomorphisms
$\{\phi_m(w_m), \phi_n(w_n) : w_m,w_n \in \Delta\}$ generates a dense subgroup of $Diff(S^1)$ (with the standard $C^{\infty}$ Lie group structure).

(b) If $\sigma$ is a diffeomorphism of finite type, then $\sigma$ has a unique factorization
$$\sigma=\lambda\circ\phi_{i_n}(w_{i_n})\circ..\circ \phi_{i_1}(w_{i_1})  $$
where $\lambda \in S^1$ is a rotation, $w_{i_j} \in\Delta\setminus\{0\}$, $j=1,..,n$, and
 $i_j\ne  i_{j+1}, j = 1,..n-1$,
for some $n$.
\end{theorem}

This is proven in Section \ref{finitetype}.

\begin{remark}\label{remark1} (a) $\phi_n$ satisfies the reality condition $\phi_n(z^*)=\phi_n(z)^*$ (where $z\to z^*$
is complex conjugation) if and only if $w_n$ is real. For the subgroup of diffeomorphisms of finite type which satisfy this reality condition (and which can be identified with a group of homeomorphisms of an open oriented string
parameterized by $[0,\pi]$), the theorem implies that this subgroup is isomorphic to a countable free product
$\mathbb R *\mathbb R*\mathbb R*..$.

(b) The (universal central extension of the) Lie algebra of the group of diffeomorphisms of a circle is essentially
a real form of the Virasoro algebra, which in some ways is similar to a
Kac-Moody algebra (see \cite{[7]}). However the Virasoro algebra notably lacks a (or has a trivial) Weyl group. In this light it is interesting to contrast the absence of relations in Theorem \ref{theorem1} to the relations for the algebraic groups associated to Kac-Moody algebras in \cite{KP}, which Kac and Peterson aptly refer to as `analytic continuation of the relations for the Weyl group'.
\end{remark}

\subsection{Algebraic Homeomorphisms}

The set of diffeomorphisms (homeomorphisms) $\sigma$ which are algebraic, i.e. satisfy a polynomial equation $p(z,\sigma(z))=0$, forms a subgroup $AlgDiff(S^1)$ ($AlgHomeo(S^1)$, respectively) of $Homeo(S^1)$. Via a linear fractional transformation which interchanges $S^1$ and $\mathbb R \cup\{\infty\}$, $y=x^3$ corresponds to an algebraic homeomorphism which is not a diffeomorphism.

A diffeomorphism $\sigma=\sigma(z)$ of finite type can be expressed in terms of radicals and hence is an algebraic diffeomorphism. The converse is false. A famous example of an algebraic diffeomorphism which cannot be expressed in terms of radicals is the `Bring Radical', which is a real solution of $y^5+y-x=0$ (with $\mathbb R$ in place of $S^1$). In turn $y=x^3+x$ corresponds to a rational diffeomorphism with an inverse which can be expressed in terms of radicals, and it is not of finite type.

\begin{proposition}\label{algebraic}The group of diffeomorphisms of finite type is properly contained in the group of diffeomorphisms which, together with their inverses, can be expressed in terms of radicals. In turn this group is properly contained in $AlgDiff(S^1)$, and in turn this group is properly contained in $AlgHomeo(S^1)$.
\end{proposition}

\begin{question} How does one characterize the group of diffeomorphisms of finite type, and is there an effective procedure for finding
the factorization in (b) of Theorem \ref{theorem1}?\end{question}

From our point of view, this question is analogous to the question of how to characterize complex trigonometric polynomials on the circle. From a complex perspective trigonometric polynomials are the restrictions of meromorphic functions on the sphere which are regular on the complement of $\{0,\infty\}$. We will adopt this complex perspective.

An algebraic homeomorphism $\sigma$ can be viewed as a multi-valued function on the sphere, or more elegantly as an equivariant meromorphic function defined on an associated Riemann surface with real structure. Finite type diffeomorphisms have the property that (viewed as multi-valued functions) they stabilize (what we will call) the Hardy decomposition
\begin{equation}\label{Hardy}\mathbb P^1=\Delta\sqcup S^1\sqcup \Delta^*\end{equation}
It is tempting to believe that this property might characterize the subgroup $FTypeDiff$ of $AlgDiff(S^1)$. For example we will show that a rational homeomorphism has this property if and only if it is a linear fractional transformation, hence is of finite type. However this is too simplistic. Homeomorphisms of the form $B^{1/n}$, where $B$ is a Blaschke product with $n$ factors, stabilize (\ref{Hardy}), and they are not generally of finite type; in deference to \cite{EKS} and \cite{Younsi}, we will refer to homeomorphisms of this type as `fingerprints of (polynomial) lemniscates', or simply fingerprints. We denote the subgroup of $AlgHomeo$ which stabilize (\ref{Hardy}) by $HardyHomeo$.

Associated to an algebraic homeomorphism satisfying $p(z,w)=0$, there are two Galois groups obtained by writing $w$ ($z$) as a polynomial with coefficients in polynomials in $z$ ($w$, respectively). For a diffeomorphism which, together with its inverse, can be expressed in terms of radicals, these two Galois groups are solvable. This is possibly a characterization of this subgroup of $AlgDiff$, which we will denote by $SolvDiff$ (the group of solvable (algebraic) diffeomorphisms). The fingerprints of lemniscates of the previous paragraph are not generally solvable, hence are not generally of finite type. The various classes of homeomorphisms which we have discussed can be displayed as
$$\begin{matrix} & & AlgHomeo & & \\ & \nearrow & \uparrow & \nwarrow & \\ Fingerprints\subset HardyHomeo &\leftarrow & FTypeHomeo & \rightarrow & SolvHomeo \\
 & & \uparrow & & \\ PolyHomeo\subset RatHomeo& \leftarrow &PSU(1,1) & \rightarrow & RatHomeo^{-1}
\end{matrix} $$
(We apologize for the clumsy notation). Our best guess is that the group of diffeomorphisms of finite type is the intersection of the subgroups of $AlgDiff$ which stabilize (\ref{Hardy}) and are solvable. The basic idea is to use the Holder series for the Galois group associated to a homeomorphism in this intersection to obtain the factorization in Theorem \ref{theorem1}. This is unresolved.

It seems plausible that the $\phi_n(w_n)$ ($w_n\ne 0$) and the rational homeomorphisms of $S^1$ corresponding to the polynomial homeomorphisms $y=x^p$, for odd prime $p$, might generate $SolvHomeo$, with no additional relations beyond those in Theorem \ref{theorem1}. We do not see an obvious candidate for a minimal set of generators to obtain all of $AlgHomeo(S^1)$.

\subsection{Diffeomorphisms}\label{introdiffeomorphisms}

Theorem \ref{theorem1} is a unique factorization result for homeomorphisms of finite type.
In the rest of the paper we are interested in factorization for more robust groups
of homeomorphisms of the circle, and for semigroups of increasing functions on the line.
This involves ordering and taking limits. There will be obvious similarities with linear
Fourier series (with the added complication that we must choose an ordering of the modes)
and with the theory of Verblunsky coefficients.

Fix a permutation (or ordering) of the natural numbers, $p:\mathbb N\to \mathbb N:n\to n'$.
Given a sequence $w = (w_n)\in \prod_{n=1}^{\infty}\Delta$, define
\begin{equation} \sigma_N = \phi_{N'}\circ.. \circ \phi_{1'} \in Diff(S^1)\end{equation}
More explicitly (in particular to emphasize the dependence on parameters)
\begin{equation}\label{factor1} \sigma_N(p,w;z) = z
\prod_{n=1}^N\frac{
(1 + \bar{ w}_{n'}\sigma_{n-1}(z)^{-n'})^{1/n'}}{
(1 + w_{n'}\sigma_{n-1}(z)^{n'})^{1/n'}},\quad \vert z\vert = 1\end{equation}
If $\sum_{n>0}\frac{1}{n}\vert w_n\vert<\infty$ (a condition which does not depend on $p$), then the product (\ref{factor1}) converges absolutely as $N\to \infty$, and
hence the limit is a degree one surjective continuous function $S^1 \to S^1$. It turns out to be a delicate
matter to determine when this limit is an invertible function, hence a homeomorphism of $S^1$; we will address this in the next subsection. We first consider a kind of core result, where invertibility is a minor issue.

\begin{theorem}\label{corelemma} Fix a permutation $p$ as above. For $s=1,2,..$, if $w\in \prod_{n=1}^{\infty}\Delta$
and $\sum_{n>0} n^{s-1}\vert w_n\vert < \infty$,
then for $z \in S^1$ the limit
$$ \sigma(p,w; z) = z\prod_{n=1}^{\infty}\frac{
(1 + \bar w_{n'}\sigma_n(z)^{-n'})^{1/n'}}{
(1 + w_{n'}\sigma_n(z)^{n'})^{1/n'}}$$
exists and $\sigma(z) = \sigma(p,w; z)$ is a $C^s$ homeomorphism of $S^1$ with $C^s$ inverse.
\end{theorem}

To prove this we will use the inverse
function theorem to show that the inverse of $\sigma$ exists and has the same degree of smoothness
as $\sigma$. In general
$$\sigma_N^{-1}=\phi_{1'}(-w_{1'})\circ\phi_{2'}(-w_{2'}) \circ..\circ \phi_{N'}(-w_{N'})$$
This does not have an expression analogous to (\ref{factor1}) which is as useful in understanding
convergence (because the composition is now growing to the right).

This leads to a basic stumbling block.

\begin{question} \label{coreconjecture} Fix a permutation $p$ as above. Is the map
$$S^1 \times \left(\mathbf c^{\infty}\cap \prod_{n=1}^{\infty}\Delta \right)\to Diff(S^1):(\lambda;w) \to \lambda\sigma(p,w; z)$$
a bijection, where $ \mathbf c^{\infty}$ is the Frechet space of rapidly decreasing sequences?
\end{question}

Although this question is unresolved, for the sake of convenience, we will refer to the parameters $w_n$ as root subgroup coordinates, relative to the ordering $p$ (see Section \ref{background} for the origin of the terminology).

The subproblem of whether this map is locally a bijection onto a neighborhood of the identity should obviously be compared to
the corresponding problem for the exponential map, which is a standard counterexample to the inverse function theorem in a Frechet space setting (see e.g. Section 3.3 of \cite{PS}).

\subsection{Less Regular Homeomorphisms}

A pivotal problem is to identify a robust criterion for $\sigma(w)$ to be invertible. A glance at the formula (\ref{derivative}) below for the derivative of $\Sigma$ suggests that $w\in l^2$ might be such a criterion.

\begin{theorem}\label{dalthorp1} If $w\in l^2 \cap \prod_{n=1}^{\infty}\Delta$ and the phases of the $w_n$ are independent uniformly distributed random variables, then almost surely $\sigma(w)$ is a homeomorphism of $S^1$.
\end{theorem}

It is not true that $w\in l^2$ implies $\sigma(w)$ is surely invertible. The point of the next two subsections is to identify the edge where deterministic conditions fail.

\subsubsection{Holder Perspective}

In a Holder setting, the best deterministic result we can hope for is Theorem \ref{corelemma}: if $w\in l^1$, then $\sigma(w)$ is a $C^1$ homeomorphism, and in particular $\sigma(w)$ is invertible. A relevant fact is that for $0<s<1$, $C^{s}$ is a Banach algebra, but it is not closed with respect to composition. This means it is not possible to filter homeomorphisms, as groups, in the Holder sense for $s<1$ (see the Appendix for more background). However this does not (to our minds) fully explain why there does not exist a weaker deterministic condition which implies mere invertibility of $\sigma(w)$. The existence of the following transition seems surprising.

\begin{theorem}\label{dalthorp2} (a) If $\limsup\limits_{n \rightarrow\infty} (n |w_n|) < \frac{1}{2}$, then $\sigma$ is invertible.

(b)	If $w_n>0$ and $\liminf\limits_{n\rightarrow\infty}(n w_n) > \frac{1}{2}$, then $\sigma$ is not invertible.
\end{theorem}

Note that $w_n=1/n$ is very comfortably $l^2$, and part (b) implies that $\sigma(w)$ is not invertible (It is also possible to check this numerically in a convincing way). The moral is that a special alignment of phases can prevent invertibility. There is a similar phenomenon which occurs in the theory of Verblunsky coefficients (see Section 2.7 of \cite{Simon}, in particular Example 2.7.5), although the basic issues are complementary, in a very interesting way, as we will explain in Section \ref{Verblunsky}.

\subsubsection{$L^2$-Sobolev Perspective}

$L^2$-Sobolev conditions are better adapted to identify decay conditions on the parameters $(w_n)$ which are equivalent to asserting that the corresponding homeomorphisms form a group, defined by some smoothness condition. The $L^2$-Sobolev analogue of the condition $w\in l^1$ is $w\in \mathbf w^{1/2}$, where in general $\mathbf w^s:=\{w: \sum n^{2s}\vert w_n\vert^2<\infty\}$.

\begin{remark}\label{remark1} If $p>1$ and $s>1/2$, then $l^p\subset \mathbf w^{1/2}$ and $\mathbf w^{s}\subset l^1$. But neither condition implies the other. For example if $w_n=\frac{1}{nlog(n)}$, $n>1$, then $w\in \mathbf w^{1/2}\setminus l^1$, and if $w_{N}=(2/3)^{n}$ when $N=3^n$ and $w_N=0$ otherwise, then the lacunary sequence $w\in l^1\setminus \mathbf w^{1/2}$.\end{remark}

\begin{question} \label{sobolevquestion} Suppose that $s\ge 1/2$. Does there exist a bijective correspondence
\begin{equation}S^1\times \left(\mathbf w^s\cap\prod_{n=1}^{\infty}\Delta\right) \to W^{s+1,L^2}Homeo(S^1):(\lambda,w) \to \lambda\circ \sigma(p,w)\end{equation}
where the target is the group of homeomorphisms which satisfy the condition $ln(\Sigma')\in W^{s}$ (in the $L^2$ sense).
 \end{question}

The critical case $s=1/2$ is far and away the most interesting. We are lacking a proper name which reflects the importance of this group.
Whereas the group $QS(S^1)$ of quasisymmetric homeomorphisms stabilizes the critical $L^2$-Sobolev class $W^{1/2}(S^1)$, $W^{1+1/2}Homeo(S^1)$ stabilizes $W^{1/2}$ and commutes with the Hilbert transform modulo Hilbert-Schmidt operators; consequently, it is the maximal group of homeomorphisms which can be represented in the associated canonical commutation relation Fock space. Put another way, it is the maximal group of homeomorphisms which has a Virasoro extension. From another point of view, $QS(S^1)$ is the symmetry group of Bers's universal Teichmuller space (modeled on a Banach space), and $W^{1+1/2}Homeo(S^1)$ is the symmetry group of the same set equipped with a tighter topology (modeled on a Hilbert space);  see \cite{[13]}.

We have extensively experimented with Question \ref{sobolevquestion}, and its truth seems plausible, at least for tame
orderings such as $p(n)=n$. Note that Theorem \ref{dalthorp1} implies that we cannot hope to find a Sobolev condition which is weaker than $\mathbf w^{1/2}$ and which implies invertibility of $\sigma(w)$ in a sure sense.

Szego settled the Verblunsky analogue of Question \ref{sobolevquestion} in an exemplary way; see Corollary \ref{Szego3} in Section \ref{Verblunsky}.

\subsubsection{Sharpness of Theorem \ref{dalthorp1} }

It turns out that Theorem \ref{dalthorp1} is relatively sharp. For example if $|w_n|=\frac{1}{\sqrt{n}}$ and the phases of the $w_n$
are i.i.d. and uniform, then $\sigma(w)$ exhibits Cantor-like behavior (the general statement is relatively complex; see Subsection
\ref{subsectCantor-like}). For deterministic magnitudes and random phases, this pins down the transition from invertibility to non-invertibility to a relatively narrow window.

\subsection{Random Magnitudes and Phases}\label{applications}

In this subsection we will allow both the magnitudes and the phases of the $w_n$ to be random. To put this in the proper perspective,
we will slightly digress.

There are a number of known interesting probability measures on $Homeo(S^1)$, with diverse origins (e.g. see \cite{AMT2}, \cite{AJKS}, \cite{KS}, part IV of \cite{[8]}, and references). One example is related
to Werner's work on conformally invariant measures on self-avoiding loops on Riemann surfaces. In this case it is of interest
to consider the welding map from topologically nontrivial self-avoiding loops in the punctured plane to homeomorphisms of $S^1$,
\begin{equation}\label{weldingmap}W:Loop^1(\mathbb C\setminus\{0\}) \to Homeo(S^1):\gamma \to \sigma(\gamma):=\phi_-^{-1}\circ \phi_+\end{equation}
where $\phi_{\pm}$ are appropriately normalized uniformizations for the regions interior and exterior to $\gamma$, respectively,
and the image of Werner's measure with respect to this map (see the Introduction to \cite{Chavez} for more detail, and references).
We are not aware of any (e.g. Poisson) geometrical structure which suggests that the image of Werner's measure, or any other natural measure, is a product in terms of the parameters $(w_n)$. However all of the measures alluded to above are related to the critical exponent $s=1/2$; at least in a heuristic sense the group of $W^{1+1/2,L^2}$ homeomorphisms is analogous to a Cameron-Martin type group for these measures. In any event it is interesting to reconsider the invertibility question of the previous subsection in a fully probabilistic background.

Consider a probability measure on $\prod_{n=1}^{\infty}\Delta$ of the form
\begin{equation}\label{mu}\prod_{n=1}^{\infty}\frac{a(n)+1}{\pi} (1-\vert w_n\vert^2)^{a(n)}\vert dw_n\vert\end{equation}
where $a(n)/n\to \beta_0>0$ as $n\to\infty$.

\begin{lemma} With respect to the probability measure (\ref{mu}),
almost surely $w\in \cap_{\epsilon>0}\mathbf w^{-\epsilon}$ and $w\notin l^2$
\end{lemma}

\begin{question}Is $\sigma(p,w)$ almost surely invertible?
\end{question}

In Section \ref{Verblunsky} we consider a Verblunsky analogue of (\ref{mu}), which is not fully understood. In the Verblunsky case it appears that there are two phase transitions. So the answer to this last question is probably complex.

\subsection{Increasing Functions on the Line}\label{increasingfunctions}

We continue to fix a permutation $p$ of $\mathbb N$. We now propose to simply ignore the invertibility question for the limit of the $\sigma_N$, and attempt to imitate the theory of Verblunsky coefficients. In the theory of Verblunsky coefficients,
there is a map
\begin{equation}\label{extendedmap}\prod_{n=1}^{\infty}\Delta \to Prob(S^1)\end{equation}
which induces a homeomorphism from a compactification of the domain to $Prob(S^1)$ with its weak$^*$ topology (see Section \ref{Verblunsky} for background). Our map is less well-behaved ``at infinity".

Suppose that $w\in \prod_{n=1}^{\infty}\Delta$. Write
$$\sigma_N(p,w;e^{i\theta}) = e^{i\Sigma_N(p,w;\theta)}$$
where the lift $\Sigma_N$ is a homeomorphism of $\mathbb R$ satisfying
$$\Sigma_N(\theta + 2\pi) = \Sigma_N(\theta) + 2\pi;$$
$\Sigma_N$ is uniquely determined modulo $2\pi\mathbb Z$. To fix a choice we take
\begin{equation}\label{liftphin}\Phi_n(w_n;\theta):=\theta-\frac2n \Theta(1+w_n e^{in\theta})\end{equation}
where $-\frac{\pi}{2}<\Theta<\frac{\pi}{2}$ is the polar angle. This definition actually makes sense
for $w_n\in D$, the closed unit disk; when $|w_n|=1$, $\Phi_n$ is a step function.
Analogous to (\ref{factor1}),
$$\Sigma_N(\theta):=\Phi_{N'}\circ ...\circ \Phi_{1'}(\theta)= \theta-2\sum_{n=1}^N\frac{1}{n'}\Theta(1+w_{n'}\sigma_{n-1}(e^{i\theta})^{n'})$$
By the chain rule
\begin{equation}\label{derivative}\frac1{2\pi}d\Sigma_N(\theta) =
\left(\prod_{n=1}^N\frac{1-\vert w_{n'}\vert^2}{\vert 1+w_{n'}\sigma_{n-1}(z)^{n'}\vert^2}\right)\frac{d\theta}{2\pi}\end{equation}
This (normalized) differential can be interpreted as a probability measure on $S^1$.

Since $S^1$ is compact, $Prob(S^1)$, the convex set of probability measures with the $\text{weak}^*$ topology relative to $C^0(S^1)$, is compact. Consequently the sequence of probability measures $(\frac{1}{2\pi}d\Sigma_N)$ has $\text{weak}^*$ limits in $Prob(S^1)$.
The question is whether there exists a unique limit. For a sequence of $w_n\in S^1$, it can easily happen that there is
lack of uniqueness of the $\text{weak}^*$ limit. We focus on the typical case.

\begin{theorem}\label{increasingfunction} Fix a permutation $p$ as above. Given $ w\in
\prod_{n=1}^{\infty}D$, if the phases of the $w_n$ are independent and uniform, then
$\frac{1}{2\pi}d\Sigma_N$ has a unique $\text{weak}^*$ limit in $Prob(S^1)$.

\end{theorem}

\subsection{Ordering of Factors and a Missing Plancherel Formula}

Is there anything special about the obvious ordering of factors, $p(n)=n$? One would suspect that restrictions on ordering would be important for almost sure type questions. We do impose a restriction in Subsection \ref{subsectCantor-like}, but it is not clear this is essential. In the theory of root subgroup factorization for loop groups, there is a need for ordering, and there are special orderings, related to factorization in the associated Weyl group. But the special Kac-Moody algebra structure of the loop group setting is lacking in our context.

In root subgroup factorization for loop groups, the analogue of the Plancherel formula is an exact factorization for Toeplitz determinants, or in representation theoretic terms, fundamental matrix coefficients; see e.g. \cite{[9]} and \cite{PP}. This is the key ingredient in the proof of the loop group analogue of Question \ref{sobolevquestion}. There are natural analogues of Toeplitz operators (with a composition operator in place of multiplication operator, see e.g. Subsection \ref{trifact}) and highest weight representations in the present context, but the corresponding determinants and matrix coefficients do not seem to factor exactly, and we do not know how
to control the (what appear to be small) correction terms.

\subsection{Plan of the Paper} In the first section we recall some basic facts about the
Virasoro algebra and group. The Virasoro point of view explains why it is natural to consider the $\phi_n$ as basic building blocks.
We also briefly mention triangular factorization (i.e. conformal welding). Triangular structure does not play a large explicit role in this paper, but our expectation is that it will play a critical role in understanding the inversion question and a (missing) Plancherel formula.

In Section \ref{alghomeo} we introduce the basic structure associated to an algebraic diffeomorphism,
and in Section \ref{finitetype} we will prove Theorem \ref{theorem1}. We will see that (b) of Theorem \ref{theorem1} can be restated in the following way: the group of diffeomorphisms of finite type is the amalgam (i.e. the free product modulo the
rotation subgroup intersection) of the covering groups $PSU(1,1)^{(n)}$, $n=1,2,..$, of
$PSU(1,1)$, the group of linear fractional transformations which stabilize $S^1$.

In Section \ref{alghomeoII} we discuss a few elementary facts about algebraic homeomorphisms and triangular factorization. This is a  classical topic which has recently received a huge boost from \cite{EKS} and \cite{Younsi}.

In Section \ref{proofofa} we prove Theorem \ref{corelemma} (rapidly decreasing coefficients map to smooth homeomorphisms). In this paper we will not address the existence of an inverse map, Questions \ref{coreconjecture} and \ref{sobolevquestion}. In a first version of this paper (on the ArXiv), we outlined an idea of proof, but we have not completed this (the main missing piece of the analytical part of the paper).

In Section \ref{cdfs} we consider increasing functions on the line. In particular we prove
Theorem \ref{increasingfunction} (there is an almost sure map, with respect to random phases, from arbitrary coefficients to increasing functions), and we address other issues which we touched on above in Subsection \ref{increasingfunctions}.

In Section \ref{invertibilitysection} we consider the issue of invertibility of $\sigma(w)$. In the first subsection we consider
deterministic conditions and prove Theorem \ref{dalthorp2}. In subsections 2 and 3 we consider random phases. In subsection 2 we prove Theorem \ref{dalthorp1} ($l^2$ is a sufficient condition for invertibility, in the presence of random phases). In subsection 3 we show that this result, Theorem \ref{dalthorp1}, is sharp.

In Section \ref{Verblunsky} we have included some remarks on how the coefficients $w_n$ compare with Verblunsky
coefficients (denoted $\alpha_n$) from the theory of orthogonal polynomials. For example suppose that $w,\alpha\in l^2 \cap \prod_{n>0}\Delta$. As we have observed above, $\Sigma(w)$ is continuous, hence the corresponding measure $\frac{1}{2\pi}d\Sigma$ does not have atoms, but its support can be a proper subset of $S^1$. By contrast, for the Verblunsky coefficients $\alpha$, the support of the corresponding measure is $S^1$, but the measure can have atoms.

In an Appendix we recall some basic smoothness conditions for homeomorphisms
of $S^1$, from a group theoretic point of view.

\subsubsection{Acknowledgement} We thank Frank Jones for sharing enlightening examples of increasing functions, and we thank
Pavel Gumenyuk for helpful correspondence and references.

\subsection{Basic Notation} Homeomorphisms of $S^1$ are assumed to be orientation
preserving, unless stated otherwise. Given a homeomorphism $\phi$ of $S^1$, there is a
homeomorphism $\Phi$ of $\mathbb R$ such that
$$\phi(e^{i\theta}) = e^{i\Phi(\theta)}$$
$\Phi$ satisfies
\begin{equation}\label{pseudo}\Phi(\theta + 2\pi) =\Phi(\theta) + 2\pi\end{equation}
and is uniquely determined up to the addition of a multiple of $2\pi$. The set of
homeomorphisms $\Phi$ of $\mathbb R$ satisfying $(\ref{pseudo})$ is a realization of the universal covering
group
$$0 \to 2\pi\mathbb Z \to \widetilde{Homeo}(S^1) \to Homeo(S^1) \to 0 $$
where $\Phi$ projects to $\phi$. In the case of $\phi=\phi_n(w_n)$ we singled out a preferred lift $\Phi_n$
in (\ref{liftphin}), and relative to a fixed permutation $p:n\to n'$, this determines a preferred lift
for $\phi=\sigma_N(p,w)$, $\Sigma_N=\Phi_{N'}\circ ...\circ \Phi_{1'}$.

A simple but important observation is that  $\widetilde{Homeo}(S^1)$ is convex, i.e.
if $\Psi_1,..,\Psi_n\in \widetilde{Homeo}(S^1)$ and $\lambda_i$, $1\le i\le n$ are nonnegative and sum to one,
then $\sum \lambda_i\Psi_i\in \widetilde{Homeo}(S^1)$. For example (a branch of) $B^{1/n}$, where $B$ is a Blaschke product with $n$ factors, represents a homeomorphism, because it can be written as
$$B^{1/n}=exp(2\pi i \sum_{j=1}^n\frac1n \Phi_1(w_{1j}))$$
for some $w_{1j}\in \Delta$.

We use $s\ge 0$ to denote order of smoothness, in various senses. If $s = k$, where
$k = 0, 1, 2,..$, then $C^s$ is the space of functions $f$ on $S^1$ such that $f$ is $k$-times
continuously differentiable. If $s = k +\alpha$, where $k=0,1,2,..$ and $0<\alpha< 1$, then
$C^s=C^{k,\alpha}$ is the space of functions $f$ on $S^1$ such that $f$ is $k$-times differentiable
and $f^{(k)}$ satisfies a Holder condition of order $\alpha$. For $s\ge 0$, $C^s$ is a Banach algebra with respect
to the norm
$$|f|_{C^s}=|f|_{C^k}+\sup_{\theta\ne \theta'}\frac{|f^{(k)}(\theta)-f^{(k)}(\theta')|}{|\theta-\theta'|^{\alpha}}$$
$C^s$ is a decomposing algebra for nonintegral $s$, i.e. if $f=\sum f_nz^n\in C^{s}$, then $f_+:=\sum_{n\ge 0}f_nz^n\in C^s$)
(see page 60 of \cite{CG}).

Define $C^{0+}:=\lim_{\alpha\downarrow 0} C^{\alpha}$, viewed as an inductive limit of Banach algebras. $C^{0+}$ is
additionally closed with respect to composition (which is not true for $\alpha$ fixed), and $C^{0+}$ is also a
decomposing algebra.

$ W^s = W^{s;L^2}$ is the space of functions $f$ on $S^1$ which are $L^2$
Sobolev of order $s$.

$[Leb]$ denotes the class of Lebesgue measure.

$(m,n)$ denotes the greatest common divisor of positive integers $m,n$.

\section{Background}\label{background}

\subsection{The Virasoro Algebra} The group of diffeomorphisms of $S^1$ (or more generally, any compact manifold) is a (nonanalytic) Frechet Lie group. The Lie algebra of $Diff(S^1)$ can be identified
with smooth real vector fields on $S^1$, with the negative of the traditional differential
geometric bracket (see \cite{Milnor}). The complexification of this Lie algebra has a universal
central extension by $\mathbb C$. The complex Virasoro algebra is the universal central
extension of the Lie subalgebra of complex trigonometric vector fields on the circle.
As a vector space

$$Vir=(\sum_{n\in\mathbb Z}\mathbb C L_n) \oplus \mathbb C\kappa$$

where $$L_n = ie^{in\theta}\frac{d}{d\theta} = -z^{n+1}\frac{d}{dz} $$

The bracket is determined by the relations
\begin{equation}\label{bracket} [L_n,L_m] = (m-n)L_{n+m} +
\frac{1}{
12}n(n^2-1)\delta(n + m)\kappa; \quad [L_n,\kappa] = 0 \end{equation}
The Virasoro algebra has a triangular decomposition, in the technical sense of \cite{[7]},
$$Vir=\mathfrak n^-\oplus\mathfrak h\oplus\mathfrak n^+, \text{ where }\mathfrak n^{\pm} =\sum_{\pm n>0}\mathbb C L_n \text{ and }\mathfrak h = \mathbb C L_0 \oplus \mathbb C\kappa$$

\begin{remark}\label{firstremark} (a) For many purposes of this paper, the reader can ignore the central extension. The embeddings below
can be viewed simply as embeddings into vector fields of the circle, and so on. But for some purposes
the extension is essential. To distinguish the embeddings we will use hats (e.g. $\widehat i$) when we are mapping into
the central extension (i.e. the Virasoro algebra), and we will drop the hat when we are mapping into vector fields on the circle
(i.e. the Witt algebra).

(b) The roots for the action of $\mathfrak h$ on $Vir$ are of the form $n\alpha_1$, $n\in\mathbb Z$, where $\alpha_1(L_0)=1$,
$\alpha(\kappa)=0$. $\alpha_1$ is the unique simple positive root.
\end{remark}

For each $n > 0$, there is a root subalgebra homomorphism corresponding to $n\alpha_1$
$$d\widehat i_n : sl(2,\mathbb C) \to Vir :\left( \begin{matrix}
0 &0\\
1& 0
\end{matrix}\right)
\to f_n =-\frac1{n}
L_{-n},$$
$$\left( \begin{matrix}
1 &0\\0 &-1\end{matrix}\right)\to h_n =\frac{2}{n}L_0 -\frac{1}{12n}(n^2- 1)\kappa,\text{ and }
\left( \begin{matrix}0 &1\\0 &0\end{matrix}\right)\to e_n =\frac{1}{n}L_n$$
The restriction of $d\widehat i_n$ to $su(1,1)$ is given by
\begin{equation}\label{suembed} di_n :\left(\begin{matrix}i &0\\0&-i\end{matrix}\right)\to ih_n, \quad\left(\begin{matrix}0 &1\\1 &0\\\end{matrix}\right)\to \frac1n L_{n} -\frac1n L_{-n}, \text{ and } \left(\begin{matrix}0 &i\\-i &0\end{matrix}\right)\to\frac {i}{n}L_n + \frac{i}{n}L_{-n}\end{equation}

\begin{remark} In a purely heuristic way, if one thinks of $z\to z^n$ as a Weyl group element, one can perhaps think of $d\widehat i_n$ as a Weyl group conjugate of $d\widehat i_1$. But the thing to note is that this heuristic Weyl group element does not flip a positive root to a negative root, because of (b) of the previous Remark \ref{firstremark}. This is a crucial structural difference between a Kac-Moody algebra and the Virasoro algebra, especially for the purposes of this paper. \end{remark}

\subsection{The Virasoro Group} The group $Diff(S^1)$ has a universal central extension
$$0 \to \mathbb Z \times i\mathbb R \to\widehat{Diff(S^1)} \to Diff(S^1) \to 0$$
Bott observed that the group $\widehat{Diff(S^1)}$ can be realized in the following
explicit way. As a manifold
$$\widehat{Diff(S^1)} =\widetilde{Diff(S^1)}\times i\mathbb R$$
In these coordinates the multiplication is given by
$$(\Phi; it) \cdot (\Psi; is) = (\Phi\circ \Psi; it + is + iC(\phi;\psi))$$
where $C$ is the $\mathbb R$-valued cocycle given by
$$C(\phi;\psi) =\frac1{48\pi}Re\int_{S^1}log(\frac{\partial\phi}{\partial z}\circ \psi )d(log(\frac{\partial\psi}{\partial z}
))$$

The corresponding Lie algebra is the real form of (the smooth completion of) $Vir$
which as a vector space equals $vect(S^1) \oplus i\mathbb R$ with the bracket given by (\ref{bracket}).

\begin{proof} One obtains the
corresponding Lie algebra cocycle via
$$c(\vec{\xi},\vec{\eta}) =\frac{\partial}{\partial s\partial t}\vert_{s=t=0}(C(e^{s\vec{\xi}},e^{t\vec{\eta}})-
C(e^{t\vec{\eta}},e^{s\vec{\xi}}))$$
$$=\frac{i}{24\pi}\int_{S^1}\frac{\partial \xi}{\partial z}d(\frac{\partial \eta}{\partial z})=\frac{i}{24\pi}\int_0^{2\pi}(\tilde{\eta}'''(\theta)+\tilde{\eta}'(\theta))\tilde{\eta}(\theta)d\theta$$
where $\vec{\xi}=\xi(z)\frac{d}{dz}=\tilde{\xi}(\theta)\frac{d}{d\theta}$. This gives the commutation relations in (1.1).\end{proof}

There are Lie group embeddings (root subgroup homomorphisms)
$$\begin{matrix} \widetilde{PSU(1,1)}& \stackrel{\widehat i_n}{\rightarrow} &  \widehat{Diff(S^1)}\\
\downarrow &  & \downarrow \\ PSU(1,1)^{(n)} &\stackrel{i_n}{\rightarrow} & Diff(S^1)\end{matrix}$$
corresponding to the Lie algebra embedding (\ref{suembed}), and the corresponding map into vector fields, where
$\widetilde{PSU(1,1)^{(n)}}$ denotes the universal covering of $PSU(1,1)$. We will
write down the embedding $i_n$
in an explicit way in the next subsection. At the level of diffeomorphisms, it
is understood geometrically as follows. The group of projective transformations of the Riemann sphere
which map the circle to itself is $PSU(1,1) \subset PSL(2,\mathbb C)$, where
$$\left(\begin{matrix}\alpha&\beta\\ \bar{\beta}&\bar{\alpha}\end{matrix}\right)\cdot z'=\frac{\bar{\eta}+\bar{\alpha}z'}{\alpha+\beta z'}$$
For $n \ge 1$ there is an n-fold covering map,
$$S^1 \to S^1 : z \to z' = z^n$$
The diffeomorphisms of $z$ which cover the projective transformations of $z'$ form a
group $PSU(1; 1)^{(n)}$, which is a realization of the n-fold covering
\begin{equation} 0 \to \mathbb Z_n \to PSU(1,1)^{(n)} \to PSU(1,1) \to 0 \end{equation}

In \cite{Ghys} it is conjectured that every finite dimensional
closed subgroup of $Homeo(S^1)$ is contained in a conjugate of one of the subgroups $PSU(1,1)^{(n)}$.

\subsection{Triangular factorization}\label{trifact} To better understand $PSU(1,1)^{(n)}$, and for
other purposes, we recall the analogue of triangular factorization for homeomorphisms
of $S^1$, often referred to as conformal welding. Just as an invertible matrix
may not have an LDU factorization, a general homeomorphism may not have a triangular
factorization; unlike the matrix case, the existence of a triangular factorization
does not imply that the factorization is unique. However for homeomorphisms
which are quasisymmetric (a relatively mild regularity condition, with multiple
characterizations - see Appendix B), the situation is completely straightforward.

\begin{theorem}\label{QS} Suppose that $\sigma$ is a quasisymmetric homeomorphism of $S^1$. Then
$$\sigma = l \circ ma \circ u$$
where
$$u = z(1 +\sum_{n\ge 1} u_nz^n)$$
is a univalent holomorphic function in the unit disk $\Delta$, with quasiconformal extension
to $\mathbb C$, $m \in S^1$ is rotation, $0 < a \le 1$ is a dilation, the mapping inverse to
$l$,
$$L(z) = z(1 +\sum_{n\ge 1} b_nz^{-n})$$
is a univalent holomorphic function on the unit disk about infinity $\Delta^*$, with quasiconformal
extension to $\mathbb C$, and the compatibility condition
$$mau(S^1) = L(S^1)$$
holds. This factorization is unique.\end{theorem}

For the state of the art, and especially for examples of homeomorphisms which
are not weldings, and for weldings which are not unique, see \cite{Bishop} and references in that paper.

\begin{remark}\label{inversefact}  If $\sigma$ has triangular factorization $lmau$, then the triangular factorization
of $\sigma^{-1}$ is given by
$$u(\sigma^{-1})(z)=\frac{1}{L(\frac{1}{z^*})^*}; \quad l(\sigma^{-1})(z)=\frac{1}
{U(\frac{1}{z^*})^*}; \quad m(\sigma^{-1})=m(\sigma)^*; \quad a(\sigma^{-1})=a(\sigma)$$
where $L$ and $U$ are inverse to $l$ and $u$, respectively.
\end{remark}

Suppose that $\phi\in PSU(1,1)^{(n)}$, and suppose that $\phi$ covers $\pm\left(\begin{matrix}\alpha&\beta\\ \bar{\beta}&\bar{\alpha}\end{matrix}\right)
\in PSU(1,1)$. Corresponding to the matrix triangular factorization
$$\left(\begin{matrix}\alpha&\beta\\ \bar{\beta}&\bar{\alpha}\end{matrix}\right)=
\left(\begin{matrix}1&0\\ \bar{\beta}\alpha^{-1}&1\end{matrix}\right)
\left(\begin{matrix}\alpha&0\\0 &\alpha^{-1}\end{matrix}\right)
\left(\begin{matrix}1&\alpha^{-1}\beta\\0&1\end{matrix}\right)$$
and setting $w_n =\alpha^{-1}\beta$, there is a heuristic factorization in the `complexification of $\widehat{Diff(S^1)}$'
$$\phi= exp(\frac{-\bar{w}_n}{n}L_{-n})\alpha^{\frac{2}{n}L_0-\frac{1}{12n}(n^2-1)\kappa}exp(\frac{w_n}{n}L_n)$$
To make rigorous sense of this, one approach is to use formal completions, as in \cite{[8]}, but we will avoid this.
At the level of diffeomorphisms, this can be understood rigorously as a triangular factorization, as in Theorem \ref{QS},
$$\phi=l(\phi)\circ ma(\phi)\circ u(\phi)$$
where
\begin{equation}\label{rsgformulas}u(\phi)(z)=\frac{z}{(1 + w_nz^n)^{1/n}}, \quad ma(\phi) = \alpha^{-2/n}; \quad a(\phi) = (1-\vert w_n\vert^2)^{1/n}\end{equation} (where the root $\alpha^{1/n}$ is unambiguous because we are considering the n-fold covering of
$PSU(1,1)$), and
$$L(\phi)(z)=z(1-\bar{w}_nz^{-n})^{1/n}$$
The composition is given explicitly by
\begin{equation}\label{firstfact}\phi=\left(\frac{\bar{\alpha}}{\alpha}\right)^{1/n}\phi_n(w_n;z)\end{equation}
where again the $n$th root is unambiguous because we are considering the n-fold covering of
$PSU(1,1)$. The expression (\ref{firstfact}) implies part (a) of the following lemma. Part (b) is a
straightforward calculation.

\begin{lemma}\label{factor3} (a) Each element in $PSU(1,1)^{(n)}$ can be written as
$$Rot(\theta) \circ \phi_n(w_n;z)$$
for a uniquely determined rotation and $w_n\in\Delta$.

(b)
$$\phi_n(w_n) \circ \phi_n(w_n')=e^{\frac{2i}{n}(1+w_n \bar{w}_n')}\phi_n(\phi_1(\bar{w}_n';w_n);z)$$
Thus all of the subgroups $PSU(1,1)^{(n)}$ have the rotation subgroup in common,
and the transformations $\phi_n$ (parameterized by a disk) give a natural cross section
for the projection from $PSU(1,1)^{(n)}$ to the quotient modulo rotations.
\end{lemma}

\begin{proposition}\label{composition1}  Suppose that the triangular factorization of $\phi$ is known:
$$\phi = l(\phi)ma(\phi)u(\phi)$$
Then for $\phi_1=\phi_1(w_1)$
$$L(\phi\circ \phi_1)(z) = L(\phi)(z)-(ma)(\phi)u(\phi)(\bar{w}_1), \quad \vert z\vert > 1$$
$$
(ma)(\phi\circ \phi_1) = (ma)(\phi)a(\phi_1)u(\phi)'( \bar{w}_1)$$
where $a(\phi_1)=(1-w_1\bar w_1)$ (by (\ref{rsgformulas})), and
$$u(\phi\circ \phi_1)(z) =\frac{1}{u(\phi)'(\bar{w}_1)(1-w_1\bar{w}_1)}
(u(\phi)(\phi_1(z))- u(\phi)(\bar{w}_1)), \quad \vert z\vert < 1$$
\end{proposition}

\begin{remark} (a) This shows that if we consider a composition $\sigma_2 = \phi_2 \circ\phi_1$,
it is not the case that the diagonal term factors, e.g. in general, $a(\sigma_2)\ne a(\phi_2)a(\phi_1)$.
This is disappointing, because the analogue of this in the context of loop groups is
true.

(b) There are also formulas for the triangular factorization of a composition of the form
$\phi_1\circ \phi$, because of Remark \ref{inversefact}.

\end{remark}

\begin{proof} First observe that our formulas for $u$ and $L$ do define univalent functions with
the proper normalizations (for example we have simply subtracted a constant from $L(\phi)$, so that
it is still univalent in $\Delta^*$, and has the correct kind of Laurent expansion). Thus it suffices
to check that for our formulas, for $\vert z\vert = 1$
$$L(\phi\circ \phi_1)\circ\phi\circ \phi_1(z) = (ma)(\phi\circ \phi_1)u(\phi\circ \phi_1)(z)$$
For our claimed formulas the left hand side equals
\begin{equation}\label{LHSformula}L(\phi) (\phi(\phi_1(z)))-(ma)(\phi)u(\phi)(\bar{w}_1)\end{equation}
and the right hand side equals
$$(ma)(\phi)(1-w_1\bar w_1)u(\phi)'( \bar{w}_1)\frac{1}{
u(\phi)'(\bar{w}_1)(1-w_1\bar{w}_1)}(u(\phi)(\phi_1(z))-u(\phi)(\bar{w}_1))$$
When we substitute
$$L(\phi) \circ \phi = (ma)(\phi)u(\phi)$$
into (\ref{LHSformula}), we see the left and right sides agree.
This completes the proof.
\end{proof}

In the introduction we mentioned fingerprints of polynomial lemniscates, i.e. algebraic homeomorphisms
of the form $B^{1/n}$, where $B$ is a Blaschke product with $n$ factors. These are best understood in terms of
their triangular factorizations, see \cite{EKS}.

\section{Algebraic Homeomorphisms, I} \label{alghomeo}

In this section we consider the Riemann surface associated to an algebraic diffeomorphism, and we discuss
a number of examples. This will be used in the proof of Theorem \ref{theorem1} in Section \ref{finitetype}.

\subsection{The Surface Associated to an Algebraic Diffeomorphism}

Suppose that $\Sigma$ is a connected compact Riemann surface with nonempty boundary
S (a disjoint union of circles). Let $\hat{\Sigma}$ denote the double, i.e.
$$\hat{\Sigma}=\Sigma^*\circ \Sigma$$
where $\Sigma^*$ is the adjoint of $\Sigma$, the surface $\Sigma$ with the orientation reversed, and the
composition is sewing along the common boundary $S$. Let $R$ denote the anti-holomorphic
involution (or reflection) fixing $S$. The basic example is the realization of
the Riemann sphere as the double of the closed unit disk $D$, where $R(z) =\frac{1}{z^*}$.

\begin{definition} (a) A Riemann surface with reflection symmetry (or a surface with
a real structure) is a connected compact Riemann surface $\hat{\Sigma}$ which is a double
$$\hat{\Sigma}=\Sigma^*\circ \Sigma$$

(b) A holomorphic map $f:\Sigma^*\circ \Sigma \to D^*\circ D$ is equivariant if it satisfies
$$ f(R(q)) =\frac{1}{f(q)^*}$$
and strictly equivariant if it additionally satisfies $f^{-1}(D) = \Sigma$.
\end{definition}

Suppose that $\phi$ is an analytic diffeomorphism of $S^1$. Analyticity implies that there exists
a reflection invariant domain $\Omega$ containing $S^1$ and an analytic continuation $\phi:\Omega
 \to \phi(\Omega)$ which is a conformal isomorphism. For $q\in \Omega$, this continuation will satisfy the
equivariance condition in (b) of the Definition, and the continuation is strictly equivariant in the
limited sense that $\Omega \cap\Delta$ will be mapped into $\Delta$. In general there does not exist
a maximal domain $\Omega$.

Suppose that $\phi$ is an algebraic diffeomorphism. In this event $\phi$ has an analytic continuation
to a multi-valued function on a reflection invariant domain
$\mathbb P^1 \setminus \{z_j\in\Delta, 1/z_j^*\in\Delta^*: 1 \le j \le n\}$ such that the singularities are algebraic (and off the circle), see e.g. \cite{Alhfors}, Theorem 4 of chapter 8. At one extreme, if $\phi$ is rational, then there are no branch points.
At another extreme, if $\phi=\phi_n(w_n)$, $w_n\ne 0$, then the number of branch points in $\Delta$ is $n$.

\begin{proposition}\label{prop1} Suppose that $w=\phi(z)$ is an algebraic diffeomorphism. Then there exist

(1) a compact connected Riemann surface with reflection symmetry $\hat{\Sigma} = \Sigma^*\circ\Sigma$;

(2) equivariant holomorphic maps $Z,W: \hat{\Sigma}\to D^*\circ D$;

(3) an irreducible polynomial $p$ (of two variables over $\mathbb C$) such that $p(Z,W)=0$;
and

(4) a distinguished component of $S$, denoted $S_1$, such that $Z,W : S_1 \to S^1$ are homeomorphisms,
$\phi = W \circ (Z\vert_{S_1} )^{-1}$, and hence $p(z,\phi(z))=0$ for $z\in S^1$.

\end{proposition}

\begin{proof} The Riemann surface defined by $\phi$ is the quotient of the
universal covering of the punctured sphere $\mathbb P^1 \setminus \{z_j, 1/z_j^*: 1 \le j \le n\}$ by the
group of automorphisms which fixes a single-valued lift of $\phi$, where the $z_j$ are the branch points
for $\phi$ in $\Delta$. There are other ways to describe this surface, such as by using germs of branches
for analytic continuations of $\phi$, see e.g. chapter 8 of \cite{Alhfors}.
The punctured sphere is stable with respect to reflection, so this reflection symmetry lifts to the universal covering.
Since $\phi$ is also reflection symmetric,
this descends to a reflection symmetry for the Riemann surface defined by $\phi$. Let $\tilde{Z}$ denote the projection
from this (incomplete) Riemann surface to the punctured sphere, and let $\tilde{W}$ denote a single-valued lift of $\phi$ to the surface. $\tilde{Z}$ is strictly equivariant, but $\tilde{W}$ is in general only equivariant (more concretely, as a multivalued function, $\phi$ does not in general map the disk to the disk). These functions satisfy a polynomial equation $p(\tilde{Z},\tilde{W})=0$, which we can suppose is irreducible. It is
well-known that this implies that the surface defined by $\phi$ can be
extended to a compact Riemann surface $\hat{\Sigma}$ in a unique way so that $\tilde{Z}$ and $\tilde{W}$
extend to holomorphic maps $Z$ and $W$ (This is essentially the Riemann extension theorem, see Theorem 2 of \cite{Donaldson}).

\end{proof}

\begin{proposition}\label{prop2} Suppose that $\phi$ is of finite type. Then (in reference to the preceding Proposition \ref{prop1})

(a) $Z$ and $W$ are strictly equivariant;

(b) $S$ is the inverse image of $S^1$ for both $Z$ and $W$; and

(c) $Z$ and $W$ are homeomorphisms restricted to each connected component of $S$.

\end{proposition}

\begin{proof} $\phi_n(w_n)$ is the $n$th root of a $PSU(1,1)$ linear fractional transformation. Consequently as a multi-valued function
it stabilizes the decomposition $\mathbb P^1=\Delta\sqcup S^1 \sqcup \Delta^*$. The same is true for a composition of the $\phi_n$, hence for any $\phi$ of finite type. Parts (a) and (b) follow from this.

For part (c): This is true on the distinguished component $S_1$ by (4) of the preceding Proposition for any $\phi$.
$Z$ and $W$ locally invert one another (with respect to composition), so when they are continued to other
components of $S$, they remain inverses. This implies (c).

\end{proof}

\subsubsection{Galois Groups}

Suppose that $\sigma\in AlgHomeo(S^1)$ and $w=\sigma(z)$ satisfies the irreducible polynomial equation $p(z,w)=0$. Then
$w=\sigma(z)$ and its inverse satisfy polynomial equations of the form
$$w^{n}+a_{n-1}(z)w^{n-1}+...+a_0(z)=0$$
$$z^{m}+b_{m-1}(w)z^{m-1}+...+b_0(w)=0$$
where the coefficients are rational functions. We obtain two Galois groups,
one associated to $\sigma$, one associated to $\sigma^{-1}$.

\subsection{Examples}

\subsubsection{Rational Homeomorphisms}

For a rational homeomorphism $w=R(z)$ of the circle, the associated Riemann surface is the sphere, $Z=z$,
and $W=R$. The degree of $R$ is odd, because the components of $R^{-1}(S^1)$ on each side of $S$ are reflections of one another.
If $R \notin PSU(1,1) $, then the degree of $R$ is at least three. There will be at least on component of $R^{-1}(S^1)$
in $\Delta$. One side of such a component will map to $\Delta$ and the other will map to $\Delta^*$. Thus $R$ will not preserve
the decomposition $\mathbb P^1=\Delta\sqcup S^1 \sqcup \Delta^*$.

The Galois group associated to $R$ is trivial. There are not any apparent restrictions on the Galois group for the inverse.

\subsubsection{$\phi_n(w_n)$}\label{basiccase}

Suppose $w_n\ne 0$. Then $w=\phi_n(z) $ satisfies
$$ w^n(1 + w_nz^n)-(z^n + \bar{w}_n)=0$$
The affine curve defined by this equation is smooth (the partial derivatives do not
simultaneously vanish). However, consider the homogeneous equation
$$ Z^n_
0 Z^n_2 + w_nZ^n_1 Z^n_2-(Z^n_0 Z^n_1 + \bar{w}_nZ^{2n}_0 ) = 0$$
(where $z = Z_1/Z_0$ and $w = Z_2/Z_0$), and the corresponding subvariety in projective
space. If $u = Z_0/Z_1$ and $ v = Z_2/Z_1$, then
$$u^nv^n + w^nv^n -(u^n + \bar{w}_nu^{2n}) = 0$$
The partial derivatives of the left hand side are
$$\frac{\partial}{\partial u}(LHS)= nu^{n-1}v^n- (nu^{n-1} + \bar{w}_n2nu^{2n-1})$$
and
$$\frac{\partial}{\partial v}(LHS)= nu^nv^{n-1} + w_nnv^{n-1}$$
Assuming that $n>1$, these partials vanish simultaneously at $u=v=0$, and this
is a point on the curve. Thus the projective variety defined by the homogeneous
equation is not smooth.

\begin{proposition} Suppose that $ 0 < \vert w_n \vert < 1$.

(a) the compact Riemann surface $\hat{\Sigma}$ associated to $\phi_n(w_n)$ has genus $(n-1)^2$.

(b)The anti-holomorphic involution $R$ for this surface,
$$ R(z,w) = (1/z^*,1/w^*)$$
has a fixed point set S which consists of n circles; there are $(n-1)(n-2)/
2$ holes in the
surface on each side of the fixed point set (this is the genus of $\Sigma$).\end{proposition}

\begin{remark} (a) This shows the projective variety associated to $\phi_n$ is not smoothly embedded in $\mathbb P^2$, for otherwise, using the genus formula for a projective curve (see page 219 of \cite{GH}), the genus would be $\frac12 (2n-1)(2n-2)$, where $2n$ is the degree
of the homogeneous polynomial.

(b) This should be compared with the Legendre normal form in the theory of
Jacobi elliptic functions
$$y^2=(1-x^2)(1-k^2x^2)$$
The affine curve is smooth (for $k\ne 0$), but the corresponding projective variety is
not smooth, for otherwise the genus would be $\frac12 (4-1)(4-2) = 3$, and we know the genus is $1$.
\end{remark}

\begin{proof} Consider first the equation in z,w coordinates:
$$w_n(1 + w_nz^n) -(z^n + \bar{w}_n) = 0$$
The partial derivatives of the left hand side are
$$\frac{\partial}{\partial z}(LHS)=nw_nz^{n-1}w^n- z^{n-1}$$
and
$$\frac{\partial}{\partial w}(LHS)=nw^{n-1}(1+\bar{w}_nz^{n})$$
For points on the affine curve, these are never simultaneously zero, and hence the
affine curve is smooth.
So we need to know how to compactify this smooth affine algebraic curve. These
points are
$$z =\infty,\quad  w =(\frac{1}{w_n})^{1/n}$$
and these are smooth. To see this, change $z$ to $\frac{1}{\zeta}$. The curve is then
$$w^n =
\frac{1 + \bar{ w}_n\zeta^n}{\zeta^n + w_n}$$
and this is perfectly well-behaved near $\zeta = 0$. We could alternately have used symmetry to
understand the behavior near $z = \infty$, since it is the reflection of what happens at
$z = 0$.
Consider the holomorphic map
$$z : \hat{\Sigma} \to \mathbb C\cup\{\infty\}$$
Let $\Sigma$ denote the inverse image of D, the closed unit disk at $z = 0$. We can think
of the surface
$$\hat{\Sigma}=\Sigma^*\circ\Sigma$$
as the double of $\Sigma$, where the involution $R$ is given by (2.4).
For the map z, there are 2n branch points at the roots $(-\bar{ w}_n)^{1/n}$ and their
reflections through $S^1$. The ramification index is $n-1$ at each branch point. By
the Riemann-Hurwitz relation
$$\chi(\hat{\Sigma}) = n\chi(S^2)-2n(n-1) = 2(1-(n- 1)^2)$$
implying that $genus(\hat{\Sigma}) = (n-1)^2$, and the genus of $\Sigma$, the number of holes in $\Sigma$,
is $(n-1)(n-2)/2$, since
$$genus(\hat{\Sigma}) = 2genus(\Sigma) + n-1$$

This construction is highly discontinuous at $w_n = 0$. When $w_n = 0$, the curve
degenerates to $w^n = z^n$, the Riemann sphere.\end{proof}

If $w_n\neq 0$, then the Galois group is $\mathbb Z_n$.

\subsubsection{ $\phi_n\circ\phi_m$} Suppose that $n \ne m$, $w_n,w_m \ne 0$, and $(m, n) = d$. The
equation we obtain from $w = \phi_n \circ\phi_m(z)$ is
$$(z^m + \bar{ w}_m)^{n/d}(1-w_nw^n)^{m/d}-(w^n-\bar{w}_n)^{m/d}(1 + w_mz^m)^{n/d}=0$$
On the one hand this polynomial has degree $mn/d$ in each individual variable for all $w_n,w_m \ne 0$.
Thus the degree is unchanging. On the one hand the total degree of this polynomial is generically
$2mn/d$, but the total degree decreases when $(-w_n)^{m/d}=w_m^{n/d}$. This means that the topology of the surface
$\hat{\Sigma}_{\phi_n(w_n)\circ \phi(w_m)}$ can vary with the parameters.

The values of $z \in\Delta$ at which branching occurs are
$$z^m =-\bar{ w}_m \text{ and } \phi_m(z)^n =- \bar{w}_n$$
We want to calculate the ramification for $Z$ at these branch points. For the
value $z = (\bar{ w}_m)^{1/m}$, there are n inverse images, $(z,w = \bar{w}_n^{1/n})$. By symmetry, the
ramification index must be the same at each point, hence this index equals $m/d$
at each of these inverse images. Given $z$ such that $\phi_m(z) = (\bar{w}_n)^{1/n}$, there are
$m$ inverse images, and possibly again by symmetry the index is the same at all of
them. Hence the ramification index must be $n/d$ at each point. So in a generic
situation we expect the ramification index
$$R = 2[m \cdot n \cdot m/d + n \cdot m \cdot n/d]$$
The Riemann-Hurwitz formula now implies
$$genus = 1 -\frac{mn}{d}+ m \cdot n \cdot \frac{m}{d} + n \cdot m \cdot \frac{n}{d}$$
This does not appear to simplify.

It seems to be a difficult problem to find some constructive procedure for finding the equation defined by a general
diffeomorphism of finite type.

\section{Finite Type Diffeomorphisms and Factorization}\label{finitetype}

In this section we will prove Theorem \ref{theorem1}. Because of Lemma \ref{factor3}, part (a) can be
restated in the following way.

\begin{theorem} Suppose that n and m are relatively prime. Then the subgroup
generated by $PSU(1,1)^{(n)}$ and $PSU(1,1)^{(m)}$ is dense in $Diff(S^1)$.\end{theorem}

The proof of this follows by a straightforward modification of the proof of Proposition
3.5.3 of \cite{PS} (which in turn relies on an argument that goes back to Cartan, used in his proof that
a closed subgroup of a finite dimensional Lie group is a Lie subgroup).

\begin{proof} Let $G$ denote the $C^{\infty}$ closure of the subgroup generated by $PSU(1,1)^{(n)}$
and $PSU(1,1)^{(m)}$ in $Diff(S^1)$. Let $\mathfrak g$ denote the set of vector fields $X$ such that
the corresponding one parameter group is contained in $G$. In a standard way $\mathfrak g$ is
a vector space and a Lie algebra, using
$$exp(t(X+Y)) = \lim_{n\to\infty} \left(exp(tX/n)\circ exp(tY/n)\right)^n$$
and
$$exp(t^2[X,Y]) = \lim_{n\to\infty} \left(exp(tX/n)\circ exp(tY/n)\circ exp(-tX/n)\circ exp(-tY/n)\right)^{n^2}$$

It is obvious that $\mathfrak g$ contains the Lie algebras of $PSU(1,1)^{(n)}$ and $PSU(1,1)^{(m)}$.
We claim that this, together with $(n,m) = 1$, implies that $\mathfrak g$ contains the Lie algebra of all trigonometric vector fields. To prove this, it suffices to show that if $(n,m)=1$, then the Lie algebra generated by $L_{\pm n}$ and $L_{\pm m}$ is the entire Witt algebra. The repeated adjoint action of the $L_{\pm m}$ on $L_n$ generates all $L_{n+km}$, $k\in\mathbb Z$; similarly the repeated adjoint action of the $L_{\pm n}$ on $L_m$ generates all $L_{m+ln}$, $l\in\mathbb Z$. Now $(n,m)=1$ implies that $\{km+ln:k,l\in\mathbb Z\}=\mathbb Z$. Thus the Lie algebra generated by $L_{\pm n}$ and $L_{\pm m}$ is the entire Witt algebra. This proves the claim.

It now follows that $\mathfrak g$ is dense in smooth vector fields. Since $\mathfrak g$ is $C^{\infty}$ closed, $\mathfrak g$ is the Lie algebra of all smooth vector fields. Thus all one parameter subgroups of $Diff(S^1)$ belong to $G$, and this implies $G=Diff(S^1)$.
\end{proof}

Since the intersection of $PSU(1,1)^{(n)}$
and $PSU(1,1)^{(m)}$ is the group of rotations, part (b) of Theorem \ref{theorem1} can be restated in the following way.

\begin{theorem}\label{theorem1'}The group of diffeomorphisms of finite type equals the amalgam of the subgroups
$PSU(1,1)^{(n)}$, $n=1,2,..$, i.e. it is the free product of these subgroups, modulo the obvious relations
arising from the common intersection, $Rot(S^1)$.\end{theorem}

\begin{lemma} Suppose that $\sigma=\phi_{i_n}(w_{i_n})\circ..\circ \phi_{i_1}(w_{i_1})$
where $w_i \in\Delta\setminus\{0\}$, $i = i_1,..,i_n$, and
 $i_j\ne  i_{j+1}, j = 1,..n-1$. Then the degree of $Z_{\sigma}$ and $W_{\sigma}$ equal
\begin{equation}\label{degreeformula} \prod_{j=1}^ni_j/\prod_{k=1}^{n-1}(i_k,i_{k+1})\end{equation}
In particular given a sequence w with non-vanishing terms, and $\sigma_N=\phi_N(w_n)\circ ..\circ \phi_1(w_1)$,
$$degree(Z_{\sigma_N} ) = N!$$

\end{lemma}

\begin{proof} Suppose that $n=1$, and let $m=i_1$. In this case, in subsection \ref{basiccase}, we saw that the associated maps $Z,W:\hat{\Sigma}_{\phi_{m}}\to \hat{D}$ have degree $m$. But more simply, in the terminology of chapter 8 of \cite{Alhfors}, we can view $\phi_m$ as a branch in a neighborhood of $S^1$ for the algebraic (multivalued) function
\begin{equation}\label{relation0}w=z^{1/m}\circ \phi_1(w_m)\circ z^m\end{equation}
(which happens to map $\Delta\to\Delta$, $S^1\to S^1$, and $\Delta^*\to\Delta^*$). We can calculate the degree
by choosing any point $z_0\in \Delta$ such that $\phi_m(w_m;z_0^m)\ne 0$ (e.g. $z_0=0$, because $w_m\ne 0$) and observing that there are exactly $m$ distinct values $w_0$ such that there exists a (germ of a) branch $f$ of the multivalued expression (\ref{relation0}) with $f(z_0)=w_0$. Of course we could also consider the ``inverse",
and find that given a generic $w_0$, there are $m$ corresponding points $z_0$. In any event the degree is $m$.

Similarly the composition
$\phi_{i_n}(w_{i_n})\circ ..\circ \phi_{i_1}(w_{i_1})$
(where $w_{i_j}\ne 0$ and $i_j\ne i_{j-1}$ for all $j$) is a branch in a neighborhood of $S^1$ for the algebraic function
$$w=z^{1/i_n}\circ \phi_{1}(w_{i_n})\circ z^{i_n} \circ z^{1/i_{n-1}} \circ \phi_1(w_{i_{n-1}})\circ ..\circ z^{1/i_1}\circ \phi_{1}(w_{i_1})\circ z^{i_1}$$ or as we prefer,
\begin{equation}\label{secondexp}w=z^{1/i_n}\circ \phi_{1}(w_{i_n})\circ z^{i_n/(i_n,i_{n-1})} \circ z^{1/(i_{n-1}/(i_n,i_{n-1}))} \circ \phi_1(w_{i_{n-1}})\circ ..\circ z^{1/(i_1/(i_2,i_{1}))}\circ \phi_{1}(w_{i_1})\circ z^{i_1}\end{equation}
To prove the Lemma, it suffices to showing this algebraic function has degree given by the formula (\ref{degreeformula}), as we observed in (b) of Proposition \ref{prop2}. We do this by induction on $n$. We can focus on $\Delta$, because these compositions map $\Delta$ into $\Delta$. The degree is obviously $\le (\ref{degreeformula}) $, so the point is to prove equality. We considered $n=1$ above. Suppose that $n>1$. By induction, aside from a finite number of exceptional points in $\Delta$, for $z_0\in\Delta$ a nonexceptional point, there will be exactly
\begin{equation}\label{n-1number}\prod_{j=1}^{n-1}i_j/\prod_{k=1}^{n-2}(i_k,i_{k+1})\end{equation} values $w_0\in\Delta$ such that there is a (germ of a) branch $f$ for
\begin{equation}\label{relation}w_1=z^{1/i_{n-1}}\circ \phi_{1}(w_{i_{n-1}})\circ z^{i_{n-1}/(i_{n-1},i_{n-2})} \circ ..\circ z^{1/(i_1/(i_2,i_{1}))}\circ \phi_{1}(w_{i_1})\circ z^{i_1}(z)\end{equation} such that $f(z_0)=w_0$.
For given $z_0$, the set of $w_0$ is acted upon by the $i_{n-1}$ roots of unity, and when $w_0\ne 0$ this action is free. We can perturb $z_0$ slightly if necessary, so that all of the $w_0\ne 0$ (we can do this, because the inverse
relation has the same properties, so that we can assume the $z_0$ and $w_0$ are simultaneously nonexceptional). In this case there will be $1/(i_{n-1},i_n)$ times (\ref{n-1number}) distinct values $w_1$ such that there is a (germ of a) branch $f$ for
\begin{equation}\label{relation2}w_2=\phi_{i_n}\circ z^{i_{n}/i_{n-1}}\circ \phi_{1}(w_{i_{n-1}})\circ z^{i_{n-1}/(i_{n-1},i_{n-2})} \circ ..\circ z^{1/(i_1/(i_2,i_{1}))}\circ \phi_{1}(w_{i_1})\circ z^{i_1}(z)\end{equation} such that $f(z_0)=w_1$. We can assume that $\phi_{i_n}(w_1^{i_n})\ne 0$. Then for generic $z_0$, there will be (\ref{degreeformula}) distinct values $w'$ such that there is a branch $f$ for (\ref{secondexp}) such that $f(z_0)=w'$. Thus the degree for (\ref{secondexp}) is given by (\ref{degreeformula}).

\end{proof}

\begin{remark}Note that this formula applies even if for some $j$, $i_j=i_{j+1}$, provided that $w_{i_j}\ne -w_{i_{j+1}}$.
\end{remark}

To prove Theorem \ref{theorem1'}, suppose by way of contradiction that
$$\lambda\phi_{i_n}(w_{i_n})\circ ..\circ \phi_{i_1}(w_{i_1})(z)=z, \quad z\in S^1$$
where $\lambda\in S^1$, $w_{i_j}\ne 0$, and $i_j\ne i_{j-1}$ for all $j$,
for some $n$. This extends to an equality of algebraic functions, and we can consider the degree of both sides.
Unless $n=1$ and $i_1=1$, the degree of the left hand side is not equal to $1$, the degree
of the right hand side. Thus by Lemma \ref{factor3} (or obviously), $\lambda=1$ and $w_1=0$, a contradiction.
This completes the proof of Theorem \ref{theorem1'}.

\section{Algebraic Homeomorphisms, II}\label{alghomeoII}

Let $RatHomeo(S^1)$ denote the set of rational homeomorphisms of $S^1$.
Via the linear fractional transformations
$$X(z)=i\frac{1-z}{z+z} \text{   and its inverse  }  Z(x)=\frac{i-x}{i+x}  $$
the sphere with the real structure $R(z)=1/z^*$ is isomorphic with the sphere and its real structure determined by conjugation.
In the latter realization, which we will refer to as the real point of view, rational homeomorphisms of the circle, modulo rotations, are identified with real rational homeomorphisms of the line. We let $PolyHomeo(S^1)$ denote the set of homeomorphisms of the circle
which correspond to polynomial homeomorphisms of $\mathbb R$.

\begin{proposition} (a) $RatHomeo(S^1)$ is a semigroup and $PolyHomeo(S^1)$ is a subsemigroup.
Both are graded by odd degree (as maps of the sphere).

(b) The group of rational homeomorphisms which have rational inverses is $PSU(1,1)$.

(c) The intersection of $RatHomeo(S^1)$ with the group of diffeomorphisms of finite type is also $PSU(1,1)$.

(d) An algebraic homeomorphism is quasisymmetric.
\end{proposition}

\begin{proof} (a) The first part of (a) is obvious. The degree of $R\in RatHomeo(S^1)$ is odd,
because the components of $R^{-1}(S^1)$ on each side of $S^1$ are reflections of one another (This is also obvious from
the real point of view). The fact that degree is multiplicative is well-known.

(b) The degree of a rational homeomorphism with rational inverse has to be one, hence it has to be a linear fractional transformation.

(c) If $R \notin PSU(1,1) $, then the degree of $R$ is at least three. There will be at least on component of $R^{-1}(S^1)$
in $\Delta$. One side of such a component will map to $\Delta$ and the other will map to $\Delta^*$. Thus $R$ will not preserve
the decomposition $\mathbb P^1=\Delta\sqcup S^1 \sqcup \Delta^*$. Thus $R$ cannot be of finite type.

(d) At a point where the derivative is zero, an algebraic homeomorphism will look like a power. This is locally quasisymmetric.
This implies (d).

\end{proof}

\begin{example} Consider the set of polynomial homeomorphisms of degree $2k+1$. From the real point of view,
this is in bijection with the set of polynomials of degree $2k$ which are nonnegative. This set is contracted to a point by the
homotopy $\lambda x^{2k}+(1-\lambda)p_{2k}$. The set of rational homeomorphisms of degree $2k+1$ similarly corresponds to
a space of rational functions, although it is not as clear how to describe this space (because the condition that the integral is
rational is a nontrivial constraint). Nonetheless this space is contracted
to $x^{2k}$ by the same homotopy.

It is not clear whether rational homeomorphisms are dense in $Homeo(S^1)$.\end{example}

Recall that a quasisymmetric homeomorphism $\sigma$ has a unique triangular decomposition $\sigma=lmau$.
It is obvious that if $u$ and $l$ are rational (algebraic), then $\sigma$ is rational (algebraic, respectively).
It is natural to ask about the converses. We will see that $\sigma$ is rational definitely does not imply that $u$ and $l$ are rational.
It seems unlikely that $\sigma$ is algebraic implies that $u$ and $l$ are algebraic, but this appears to be open (we hope to resolve this in the sequel).

\begin{proposition} Suppose that $\sigma$ is the inverse of a rational homeomorphism, and let
$\sigma =l ma u$ be its triangular factorization. If $u$ is rational, then $\sigma\in PSU(1,1)$.
Similarly, if $\sigma$ and $L$ are rational, then $\sigma\in PSU(1,1)$.
\end{proposition}

\begin{proof} Suppose that $\sigma^{-1}=R$ is rational and $R$ is not a linear fractional transformation. The degree of $R$
(as a map of the sphere) is at least 3. The inverse image of $S^1$ will have at least three components, hence at least one component
in each of $\Delta$ and $\Delta^*$. This implies that $R$ will map a disk in $\Delta^*$ onto $\Delta$. Thus $R$ will have a zero in $\Delta^*$. In the proof of this proposition, without loss of generality, we can suppose that $R(\infty)=0$. For otherwise we can compose $R$ with a $g\in PSU(1,1)$ so that this is the case.

Let $f_+=mau$. Then $f_+(R(z)) =L(z)$ for $z$ in an annular neighborhood of $S^1$. Since $L$ is holomorphic in $\Delta^*$, we can analytically continue the left hand side to $z=\infty$; let $g(z)$ denote this analytic continuation along a neighborhood of some path from $S^1$ to $\infty$. By way of contradiction, suppose that $f_+$ is rational. In this event, $g(z)=f_+(R(z))$. Thus $f_+(R(\infty))=f_+(0)=0$. But $L(\infty)=\infty$, a contradiction. Thus $f_+$ cannot be rational.
This proves the proposition for the inverse of a rational homeomorphism.

Suppose that $\sigma$ is rational and not a linear fractional transformation. The triangular factorization of the inverse is
given by
$$\sigma^{-1}=\frac{1}{U(\frac{1}{z^*})^*}\circ m^* a\circ \frac{1}{L(\frac{1}{z^*})^*}$$
This, and the first part of the proof, implies that $L$ cannot be rational.
\end{proof}

Another question one can ask is, given a very simple $u$, e.g. $u$ polynomial, does this imply that $l$ is algebraic? The answer is no.

\begin{example} Define
$$\sigma(z)=\frac{1+cos(\frac{\pi}{1+z})}{1-cos(\frac{\pi}{1+z})} $$
Then
$$u(z)=z(1+z/2), \quad a=8/\pi^2, \quad l(z)=\frac{1+cos(\frac{2\pi}{\sqrt{4+\pi^2z}})}{1-cos(\frac{2\pi}{\sqrt{4+\pi^2z}})} $$
\end{example}

For the general question of how to obtain $L$ from $u$, see \cite{GGV}.

\section{Diffeomorphisms: Proof of Theorem \ref{corelemma}}\label{proofofa}

We recall the statement to be proved:

\begin{theorem}\label{parta} Fix a permutation $p:\mathbb N\to \mathbb N:n\to n'$. For $s=1,2,..$ if $w\in \prod_{n=1}^{\infty}\Delta$
and $\sum_{n>0} n^{s-1}\vert w_n\vert < \infty$, then the limit
$$\sigma(p,w; z) = z
\prod_{n=1}^{\infty}\frac{
(1 + \bar{w}_{n'}\sigma_{n-1}(z)^{-n'})^{1/n'}}{
(1 + w_{n'}\sigma_{n-1}(z)^{n'}})^{1/n'}  $$ exists and defines a $C^s$ homeomorphism of $S^1$.\end{theorem}

We first consider the case $s=1$.

\begin{lemma}\label{lemma4} (a)
$$\Phi_n'(\theta)=\frac{1-\vert w_n\vert^2}{\vert 1+w_nz^n\vert^2},\quad \vert z\vert=1$$

(b)
$$\Sigma_N'(\theta)=\prod_{k=1}^N\Phi_{k'}'(\Sigma_{k-1}(\theta))=\prod_{k=1}^N\frac{1-\vert w_{k'}\vert^2}{\vert 1+w_{k'}\sigma_{k-1}^{k'}\vert^2}$$

(c) If $(w_n)$ is absolutely summable, then the product expression for $\Sigma'$,
$$\Sigma'(\theta)=\prod_{n=1}^{\infty}\frac{1-\vert w_{n'}\vert^2}{\vert 1+w_{n'}\sigma_{n-1}(z)^{n'}\vert^2}$$
is absolutely convergent on $\mathbb R$, and $\sigma$ is a $C^1$ diffeomorphism of $S^1$.\end{lemma}

\begin{proof} (a) is a straightforward calculation. Part (b) follows from the chain rule,
$$\Sigma_N'(\theta)=\prod_{k=1}^N\Phi_{k'}'(\Sigma_{k-1}(\theta))$$
and part (a).

Assuming that $(w_n)$ is absolutely summable, the expression for the
derivative of $\Sigma$ is absolutely convergent, because
$$\prod_{n=1}^{\infty}\frac{1-\vert w_{n'}\vert^2}{\vert 1+w_{n'}\sigma_{n-1}(z)^{n}\vert^2}\le
\prod_{n=1}^{\infty}\frac{1-\vert w_n\vert^2}{(1-\vert w_{n}\vert)^2}=
\prod_{n=1}^{\infty}\frac{1+\vert w_n\vert}{(1-\vert w_{n}\vert)}$$
The derivative of $\Sigma$ is positive and continuous; together with the inverse function theorem,
this implies that $\Sigma$ and its inverse are $C^1$.\end{proof}

To investigate the higher derivatives of $\Sigma$, define
$$ B_n(\theta) := ln(\Phi_n'(\theta))= ln(\frac{1-\vert w_n\vert^2}{\vert1 + w_nz^n\vert^2} )$$
\begin{equation}\label{4.2}=-ln(1 + \bar{ w}_nz^{-n}) + ln(1-\vert w_n\vert^2) + ln(1 + w_nz^n), \quad z = e^{i\theta}\end{equation}
and
$$B(\theta) := ln(\Sigma'(\theta)) =\sum_{n=1}^{\infty}B_{n'}(\Sigma_{n-1}(\theta))$$

\begin{lemma}\label{lemma5} (a) For $s=1,2,..$,
$$B_n^{(s)}(\theta)= (in)^s \frac{w_nz^nA_{s-1}(-w_nz^n)}{(1 + w_nz^n)^s}+c.c., \quad z = e^{i\theta}$$
where the $A_{s-1}$ are the Eulerian polynomials.

(b) For given $s$ there is a constant $c=c(s)$ independent of $n$ such that
$$\vert B_n^{(s)}(\Sigma_{n-1}(\theta))\vert\le cn^s\vert w_n\vert(1-\vert w_n\vert)^{-s}$$
\end{lemma}

\begin{proof} From (\ref{4.2}) (and expanding the logarithm in a power series)
$$(\frac{\partial}{\partial \theta})^sB_n(\theta)=(\frac{\partial}{\partial\theta})^s ln(1+w_nz^n)+c.c.$$
$$=\sum_{k=1}^{\infty}\frac1k (-w_n)^k(\frac{\partial}{\partial\theta})^s z^{kn}+c.c.=(in)^s\sum_{k=1}^{\infty}k^{s-1}(-w_nz^n)^k+c.c.$$
This can be summed using the basic power series identity of Euler
\begin{equation}\label{identitygrv}\sum_{k=1}^{\infty}k^nq^k=\frac{qA_n(q)}{(1-q)^{n+1}},\quad \vert q\vert<1\end{equation} where $A_n$ is the $n$th Eulerian polynomial.
This implies part (a).

Part (b) follows from (a), where we bound $\vert z^nA_{s-1}(w_nz^n)\vert$ by a constant depending only on
$s$ (and the size of coefficients for the Eulerian polynomial $A_{s-1}$), using the facts that $\vert z\vert=1$ and $\vert w_nz^n\vert<1$.

\end{proof}

We now complete the proof of Theorem \ref{parta}.

\begin{proof} We will prove the slightly broader statement that if  $\sum n^{s-1}\vert w_n\vert < \infty$,
then there is a bound for the derivatives of $B_N$ up to order $s-1$ which is independent of $N$. This will imply
that $B$ itself is $C^{s-1}$. Lemma \ref{lemma4} takes care of the case $s=1$.

Suppose $s>1$. Faa di Bruno's formula for higher derivatives of a composition of functions implies that

$$( \frac{d}{d\theta})^{s-1}B_N(\theta) =\sum_{n=1}^{N}(\frac{ d}{d\theta})^{s-1}(B_{n'}\circ\Sigma_{n-1})(\theta)$$

\begin{equation}\label{4.5}=\sum_{n=1}^{N}\sum_{k=1}^{s-1}B_{n'}^{(k)}(\Sigma_{n-1}(\theta))\mathcal B_{s-1,k}(\Sigma'_{n-1},..,\Sigma_{n-1}^{(s-1-k)})\end{equation}
where $\mathcal B_{s-1,k}$ denotes the Bell polynomial of degree $k$. For example
$$B''(\theta) =\sum_{n=1}^{\infty}\left(B_{n'}''(\Sigma_{n-1}(\theta))\Sigma_{n-1}'(\theta)^2
+B_{n'}'(\Sigma_{n-1}(\theta))\Sigma_{n-1}''(\theta)\right)$$ In general the Bell polynomials have positive integral coefficients.

Using (b) of Lemma \ref{lemma5}, we can bound the
sum in (\ref{4.5}) by

$$\sum_{n=1}^{\infty}\sum_{k=1}^{s-1}c{n'}^{k}\frac{\vert w_{n'}\vert}{(1 - \vert w_{n'}\vert)^k}B_{s-1,k}(
\sup\vert\Sigma'_{n-1}\vert,..,\sup\vert\Sigma_{n-1}^{(s-1-k)}\vert)$$

$$\le \sum_{n=1}^{\infty}c{n'}^{s-1}\vert w_{n'}\vert\sum_{k=1}^{s-1}B_{s-1,k}(
\sup\vert\Sigma'_{n-1}\vert,..,\sup\vert\Sigma_{n-1}^{(s-1-k)}\vert)$$
In this sum, because $s$ is fixed, we are considering a fixed finite number of Bell polynomials. Since the
orders of the derivatives appearing in the sum over $k$ are strictly less than $s-1$, by induction
we find a bound for
$$\sum_{k=1}^{s-1}B_{s-1,k}(
\sup\vert\Sigma'_{n-1}\vert,..,\sup\vert\Sigma_{n-1}^{(s-1-k)}\vert)$$ which is independent of $N$.
This completes the induction step.
\end{proof}

\section{Semigroup of Increasing Functions}\label{cdfs}

In this section we try to define the forward map $w\to\Sigma(w)$ as broadly as possible.
We first formalize some of the algebraic structures that are relevant when we do not insist on invertibility or continuity of $\Sigma(w)$.

\begin{definition} (a) $\widetilde{CDF}(S^1)$ is the semigroup of right continuous nondecreasing
functions on $\mathbb R$ satisfying
$$\Sigma(\theta + 2\pi) =\Sigma(\theta) + 2\pi$$
where multiplication is given by composition.

(b) $CDF(S^1)$ is the quotient of $\widetilde{CDF}(S^1)$ by the central subgroup $2\pi\mathbb Z$, where
$2\pi n$ is identified with the map $\theta \to \theta + 2\pi n$. We can identify cdfs (i.e. elements of $CDF(S^1)$)
as self-maps of $S^1$ which (in reference to the orientation) are right continuous and nondecreasing.
\end{definition}

\begin{proposition} (a) The map $CDF(S^1)\to Prob(S^1):\Sigma \to
\frac1{2\pi}d\Sigma$, the distributional derivative, induces a short exact sequence
$$0\to Rot(S^1)\to CDF(S^1) \to Prob(S^1)\to 0$$
We will refer to $\Sigma$ as a cdf corresponding to
its generalized derivative $d\Sigma$.

(b) With the weak star topology relative to $C^0(S^1)$, $CDF(S^1)$ is a topological semigroup.

(c) $Homeo(S^1)$ is the group of units for $CDF(S^1)$. It is not dense. It is not
closed.

(d) The cdfs corresponding to measures with finite support is a dense normal
subsemigroup.

(e) Fix $n$. The cdfs corresponding to measures with $n$ atoms is a normal subsemigroup.
\end{proposition}

\begin{proof} This is straightforward.\end{proof}

Let $D:=\{\vert z\vert\le 1\}$, the closed unit disk.

\begin{definition} For $w_n=u_n+iv_n=r_ne^{iq_n}\in D$,
$$\Phi_n(w_n;\theta):=\theta-\frac2n \arctan\left(\frac{u_n \sin(n\theta)+v_n \cos(n\theta)}{1+u_n \cos(n\theta)-v_n \sin(n\theta)}\right)=
\theta-\frac2n \arctan\left(\frac{r_n \sin(n\theta+q_n)}{1+r_n\cos(n\theta+q_n)}\right)$$
when $1+u_n \cos(n\theta)-v_n \sin(n\theta)\ne 0$ and extend the definition to all $\theta\in\mathbb R$ by insisting that
$\Phi_n$ is right continuous. We also define
$$\phi_n(w_n;z):=e^{i\Phi_n(w_n;\theta)},\quad z=e^{i\theta}$$
and
$$\Sigma_N(w;\theta)=\Phi_N(w_N)\circ..\circ \Phi_1(w_1)(\theta)$$
\end{definition}

This agrees with our previous definition of $\Phi_n(w_n)$ when $w_n\in \Delta$.

\begin{proposition} Suppose $w_n\in D$. (a) $\Phi_n(w_n)\in \widetilde{CDF}(S^1)$ and $\Phi_n(w_n)$ is uniquely determined by the normalized distributional derivative $\frac1{2\pi}d\Phi_n\in Prob(S^1)$.

(b) Suppose that $\vert w_n\vert=1$, i.e. $w_n=e^{iq_n}$. Then $\Phi_n$ has image consisting of the (angles corresponding to the) $\frac1n$th roots of $1/w_n=w_n^*$, i.e. the angles $\frac 1n q_n+\frac kn 2\pi$, $k=0,..,n-1$, and the points of discontinuity are the (angles corresponding to the) $\frac1n$th roots of $-1/w_n=-w_n^*$, i.e. the angles $-(\frac 1n q_n+\frac kn 2\pi)$, $k=0,..,n-1$. Thus $\Phi_n$ is a step function with the length and height of each step given by $2\pi/n$, i.e. $\frac1{2\pi}d\Phi_n$ is a sum of delta measures at the angles $-(\frac 1n q_n+\frac kn 2\pi)$, $k=0,..,n-1$, each of mass $ \frac1n $.

\end{proposition}

\begin{proof} (a) is clear for $w_n\in\Delta$. It will follow from (b) in the case $w_n\in S^1$.

Suppose that $w_n\in S^1$. Then
$$ \phi_n(z)^n=z^n\frac{1+\bar{w}_nz^{-n}}{1+w_nz^n}=\frac1{w_n}$$
This implies the first half of the first part of (b).

When $w_n\in\Delta$
$$\Phi_n'(\theta)=\frac{1-\vert w_n\vert^2}{\vert 1+w_n z^n\vert^2}$$
By letting $w_n$ tend to the circle, we see that the jumps will occur when the denominator tends to zero, which is at the $\frac1n$ the roots of $-1/w_n$.This completes the proof of (b), and hence also of (a).

\end{proof}

\subsection{Proof of Theorem \ref{increasingfunction} }

We recall the statement to be proved:

\begin{theorem}\label{increasingfn} Fix a permutation $p:\mathbb N\to \mathbb N$. Given $ w\in
\prod_{n=1}^{\infty}D$, if the phases of the $w_n$ are independent and uniform, then
$\frac{1}{2\pi}d\Sigma_N$ has a unique $\text{weak}^*$ limit in $Prob(S^1)$.

\end{theorem}

\begin{proof} Write $w_n=r_ne^{iq_n}$.  For $\theta \in \mathbb{R}$
\begin{equation}
\Sigma_N(\theta) = \theta - \sum_{n=1}^N \frac{2}{n'}\arctan\left(\frac{r_{n'}\sin(n'\Sigma_{n-1}(\theta)+q_{n'})}{1 + r_{n'}\cos(n'\Sigma_{n-1}(\theta)+q_{n'})}\right)
			\end{equation}
Fix $\theta$. The random variable $\Sigma_{n-1}(\theta)$ depends on the random variables $q_{k'}$ for $k<n$, which are independent of $q_{n'}$.
Because $q_{n'}$ is uniform, $v_{n'}: = n'\Sigma_{n-1}(\theta) + q_{n'}$ is also uniform. Therefore
\begin{equation}\Sigma_N(\theta) = \theta - \sum_{n=1}^N \frac{2}{n'}\arctan\left(\frac{r_{n'}\sin( v_{n'})}{1 + r_{n'}\cos(v_{n'})}\right)
			\end{equation}
where the $v_n$ are i.i.d. and uniform. The random variable $\arctan\left(\frac{r_n\sin(v_n)}{1 +r_n\cos(v_n)}\right)$ is an odd function of $v_n$. It follows that $\Sigma_N(\theta)$ almost surely converges as $N$ goes to infinity to the (conditionally convergent) sum
\begin{equation}
\Sigma(\theta) = \theta - \sum_{n=1}^{\infty} \frac{2}{n'}\arctan\left(\frac{r_{n'}\sin(n'\Sigma_{n-1}(\theta)+q_{n'})}{1 + r_{n'}\cos(n'\Sigma_{n-1}(\theta)+q_{n'})}\right)
			\end{equation}

Because $S^1$ is compact, the sequence of probability measures $\frac{1}{2\pi}d\Sigma_N$ has a $\text{weak}^*$ limit in $Prob(S^1)$.
Suppose that a subsequence $d\Sigma_{n_j}$ converges to $dF$. This means that $\Sigma_{n_j}$ converges pointwise to $F$ at all points of continuity of $F$. As a nondecreasing function which is right continuous, $F$ is determined by its values at a countable dense set of points. For each of these points $\theta_0$, with probability one, $F(\theta_0)=\Sigma(\theta_0)$. Since the set of points is countable, this implies that almost surely $F=\Sigma$. This implies uniqueness of the limit in the theorem.

\end{proof}

\begin{corollary} Given $ w\in
\prod_{n=1}^{\infty}D$, if the phases of the $w_n$ are independent and uniform, then almost surely
$\Sigma_N(\theta)$ converges to $\Sigma(\theta)$ at all points of continuity for $\Sigma$, hence at all but countably many points.

\end{corollary}

\section{Invertibility of $\sigma(w)$}\label{invertibilitysection}

Throughout this section $w\in \prod_{n=1}^{\infty}\Delta$, $w_n=r_ne^{iq_n}$, and we fix a ordering $p$ of $\mathbb N$.
In the first subsection we consider invertibility in a deterministic framework, and we set limits on the best possible conditions.
In the second and third subsections we consider invertibility of $\sigma(w)$, assuming that the phases of the $w_n$ are i.i.d. and uniform; in the second we show that $l^2$ is sufficient for almost sure invertibility, and in the third we show this is essentially best possible.

\subsection{Proof of Theorem \ref{dalthorp2}   }

We have established that $w\in l^1$ implies $\sigma(w)$ is invertible (in fact it is $C^1$ with $C^1$ inverse). We now show that this is the most robust Holder condition we can hope for, and in addition we identify a phase transition.

\begin{theorem}	(a) If $\limsup\limits_{n\rightarrow\infty} (n |w_n|) < \frac12$, then $\sigma(w)$ is invertible and hence a homeomorphism of $S^1$.\\
	(b) If $w_n>0$ for all $n$ and $\liminf\limits_{n\rightarrow\infty}(n w_n) > \frac{1}{2}$,
then there exists $\theta_0 > 0$ such that $\Sigma(w,\theta_0) = \Sigma(w,0) = 0$. Consequently $\sigma(w)$ is not invertible.
\end{theorem}

\begin{proof} (a) First observe that
\begin{equation}\label{bounds}|\Phi_n(w_n,\theta) - \theta|=|\frac2n \Theta(1+w_nz^n)| \le \frac{2}{n}\arcsin{r_n} \le \frac{\pi}{n}r_n
\end{equation}
This follows from $|\sin(\Theta(1+w_nz^n))|=|\Im(1+w_nz^n)|=|\Im(w_nz^n)|\le |w_n|$.

Fix $\theta\in\mathbb{R}$, and let $I = \{x | \Sigma(x) = \Sigma(\theta)\}$. By monotonicity $I$ must be an interval. We will show that under the assumption $\limsup( n|w_n|) < \frac12$, $I = \{\theta\}$, and hence $\Sigma$ is invertible.

Let $S_n(\theta) = ... \circ \Phi_{n+3} \circ \Phi_{n+2} \circ \Phi_{n+1}$. Notice that $S_n\circ\Sigma_n = \Sigma$. By (\ref{bounds}),
$$	|S_n(\theta) - \theta| \le \sum_{k=n+1}^\infty \frac2n\arcsin (r_n)	$$
In particular $|S_0(\theta) - \theta|\le \sum_{k=1}^\infty \frac2n\arcsin( r_n)$. This implies that the length of $I$ is bounded above by $2 \sum_{k=1}^\infty \frac2n\arcsin( r_n)$, and so at least one of $\theta \pm \sum_{k=1}^\infty \frac2n\arcsin (r_n)$ must lie outside of $I$. Let $I_n = \Sigma_n(I)$, then note that $I_n = \{x | S_n(x) = S_n(\theta)\}$. By similar reasoning at least one of $\theta \pm \sum_{k=n+1}^\infty \frac2n\arcsin( r_n)$ is not in $I_n$. Thus for each $n$ the following points are not in $I$:
	\begin{eqnarray}
	a_n &=& \Sigma^{-1}_n\left(\theta + \sum_{k=n+1}^\infty \frac2n\arcsin(r_n)\right)\\
	b_n &=& \Sigma^{-1}_n\left(\theta - \sum_{k=n+1}^\infty \frac2n\arcsin(r_n)\right)
	\end{eqnarray}

By monotonicity of $\Sigma_n$, $a_n > \theta > b_n$. Furthermore, observe that for the derivative of $\Phi^{-1}_n(\theta)=\Phi_n(-w_n,\theta)$
$$ \frac{1-r_n^2}{1 - 2r_n\cos(n\theta + q_n) + r_n^2} \le \frac{1-r_n^2}{1 - 2r_n + r_n^2} = \frac{1 + r_n}{1-r_n}$$
Therefore
$$	\frac d {d\theta} \Sigma^{-1}_n(\theta) \le \prod_{k=1}^n \frac{1+r_k}{1-r_k}	$$
Now choose $\frac12 >\lambda > \liminf( nr_n)$. Then
\begin{eqnarray}
	|a_n - b_n| &\le& \left(\prod_{k=1}^n \frac{1+r_k}{1-r_k}\right)\left(2 \sum_{k=n+1}^\infty \frac2k\arcsin r_k\right)\nonumber\\ &\le& \left(\prod_{k=1}^n \frac{1+r_k}{1-r_k}\right)\left(2\pi \sum_{k=n+1}^\infty \frac{r_k}{k}\right)
	\le K\left(\prod_{k=1}^n \frac{1+\frac{\lambda}{k}}{1-\frac{\lambda}{k}}\right)\left( \sum_{k=n+1}^\infty \frac{1}{k^2}\right)\nonumber
\end{eqnarray}
for sufficiently large $n$ and some appropriate proportionality constant $K$. Taking a logarithm, this becomes:
\begin{eqnarray}
	\log|a_n - b_n| &\le& \log K + \sum_{k=1}^n \left(\log(1 + \frac{\lambda}{k}) - \log(1-\frac{\lambda}{k})\right) + \log\left(\sum_{k=n+1}^\infty \frac{1}{k^2}\right)\nonumber\\
	&\sim&\log K + \sum_{k=1}^n \frac{2\lambda}{k} +\log\left(\sum_{k=n+1}^\infty \frac{1}{k^2}\right) \nonumber\\
	&\sim& A + 2\lambda\log(n) - \log(n) = A + (2\lambda - 1)\log(n)\nonumber
\end{eqnarray}
for some constant $A$. Note that this goes to $-\infty$ because $\lambda<\frac12$. Hence, $|a_n - b_n|$ goes to $0$ as $n$ goes to infinity. Since $a_n$ and $b_n$ lie above and below the interval $I$ for all $n$, we conclude that the length of $I$ is $0$, so it contains only a single point, $\theta$. Since $\theta$ was arbitrary, we conclude that $\Sigma$ is invertible.
	
Part (b): Because $w$ is real, $\Sigma(0) = 0$.	We will show that there exists a positive constant $c$ such that for $n$ sufficiently large, $\Sigma_{n-1}(\theta) < \frac{c}{n}$ implies $\Sigma_n(\theta) < \frac{c}{n+1}$. Since there is certainly a positive $\theta_0$ such that $\Sigma_{n-1}(\theta_0) < \frac{c}{n}$, this will show that $\Sigma_n(\theta_0) \rightarrow 0$ as $n\to\infty$ and hence $\Sigma$ is not invertible.

We pick $c$ in the following way. Since $\liminf( n w_n) > \frac12$, there exists $c > 0$ so that $\liminf( n w_n) > \frac{c}{2\sin( c)}$, and $c < \frac{\pi}{2}$. This implies $\frac{c}{n} < \frac{\pi}{2n} < \frac{\pi}{2n-1}$. Observe the following asymptotic inequality for $\alpha < \frac{c}{n}$:
\begin{equation}\label{asymptotic}
	\frac{\sin(\frac{n\alpha}{2(n+1)})}{\sin((1 - \frac1{2(n+1)})n\alpha)} < \frac{\sin(\frac{c}{2(n+1)})}{\sin((1 - \frac1{2(n+1)})c)}\sim\frac{c}{2\sin (c)} \frac{1}{n} \text{  as  } n\to\infty\end{equation}

In what follows we suppose that $\epsilon < \liminf (nw_n) - \frac{c}{2\sin (c)}$ and $n$ is large enough so that $n w_n > \frac{c}{2\sin (c)} + \epsilon$. Now suppose that $\Sigma_{n-1}(\theta) < \frac{c}{n}$. Using (\ref{asymptotic}), we have the following chain of implications, where we abbreviate $\Sigma_{k}(\theta)$ to $\Sigma_{k}$:
$$	w_n > \left(\frac{c}{2 \sin (c)} + \epsilon\right) \frac{1}{n} > \frac{\sin (\frac{n\Sigma_{n-1}}{2(n+1)})}{\sin((n - \frac{n}{2(n+1)})\Sigma_{n-1})}  $$
$$	w_n \sin\left(\left(n- \frac{n}{2(n+1)}\right)\Sigma_{n-1}\right) > \sin \left(\frac{n}{2(n+1)}{\Sigma_{n-1}}\right)  $$
$$	w_n \sin( n\Sigma_{n-1}) \cos \left(\frac{n\Sigma_{n-1}}{2(n+1)}\right) - w_n \cos(n\Sigma_{n-1}) \sin\left( \frac{n\Sigma_{n-1}}{2(n+1)}\right) > \sin\left(\frac{n\Sigma_{n-1}}{2(n+1)}\right)  $$
$$	w_n \cos\left(\frac{n\Sigma_{n-1}}{2(n+1)}\right) \sin( n\Sigma_{n-1}) > \sin \left(\frac{n\Sigma_{n-1}}{2(n+1)}\right) + w_n \sin\left(\frac{n\Sigma_{n-1}}{2(n+1)}\right)\cos( n\Sigma_{n-1})  $$
$$	w_n \sin(n\Sigma_{n-1}) > \tan\left(\frac{n\Sigma_{n-1}}{2(n+1)}\right)\left(1+w_n\cos( n\Sigma_{n-1}) \right) $$
$$	\frac{w_n \sin (n\Sigma_{n-1})}{1+w_n\cos( n\Sigma_{n-1}) } > \tan\left(\frac{n\Sigma_{n-1}}{2(n+1)}\right)  $$
	$$\frac{2}{n}\arctan\left(\frac{w_n \sin (n\Sigma_{n-1})}{1+w_n\cos( n\Sigma_{n-1}) }\right) > \frac{\Sigma_{n-1}}{n+1}  $$
$$	\frac{n}{n+1}\Sigma_{n-1} > \Sigma_{n-1} - \frac{2}{n}\arctan\left(\frac{w_n \sin( n\Sigma_{n-1})}{1+w_n\cos( n\Sigma_{n-1}) }\right)=\Sigma_n  $$
Thus $	\frac{n}{n+1} \Sigma_{n-1}(\theta) >\Sigma_{n}(\theta)$. We are assuming $\Sigma_{n-1}(\theta) < \frac{c}{n}$,
and hence $\Sigma_{n}(\theta)< \frac{c}{n+1}$. Applying this
to $\theta=\theta_0>0$ as above, and letting $n$ go to infinity, we obtain $\Sigma(\theta_0) = 0$. This implies that $\Sigma$ is not invertible and completes the proof of part (b).
\end{proof}

\subsection{Proof of Theorem \ref{dalthorp1}}

We recall the statement to be proved:

\begin{theorem}	Suppose that $r\in l^2$ and the phases $q_n$ are i.i.d. and uniform. Then almost surely $\sigma(p,w)$ is 1-1.
\end{theorem}

\begin{proof} Fix $\theta \in \mathbb R$.
\begin{equation}\label{derivative1}\Sigma_N'(\theta) = \prod_{n=1}^N \left(\frac{1 - r_{n'}^2}{r_{n'}^2 + 2r_{n'}\cos(n'\Sigma_{n-1} + q_{n'}) + 1}\right) = \prod_{n=1}^N \left(\frac{1 - r_{n'}^2}{r_{n'}^2 + 2r_{n'}\cos(v_{n'}) + 1}\right)
\end{equation}
where the random variables $v_{n'}=n'\Sigma_{n-1}(\theta)+q_{n'}$ are i.i.d. and uniform (see the proof of Theorem \ref{increasingfn} in the previous subsection). Because $r$ is square summable, the convergence of this product to a non-zero number as $N$ goes to infinity is equivalent to the convergence of $\sum r_{n'} \cos(v_{n'})$. Because $r\in l^2$ and the $v_n$ are i.i.d. and uniform, this sum converges.
Let $W(\theta)=\lim_{N\to\infty}\Sigma_N'(\theta)$. For each $\theta$, almost surely $W(\theta)>0$. Therefore by Fubini's Theorem,
almost surely (with respect to the random phases), $W(\theta)>0$ almost surely with respect to Lebesgue measure.

Because $w\in l^2$, we know that $\Sigma_N$ converges uniformly to $\Sigma$. Fatou's Lemma implies that for any $\delta > 0$
$$\Sigma(\theta+\delta)-\Sigma(\theta)=\lim_{N\to\infty}\int_{\theta}^{\theta+\delta} \Sigma_N'(\phi)d\phi\ge
\int_{\theta}^{\theta+\delta} W(\phi)d\phi>0$$
Therefore $\Sigma$ is invertible, and hence $\sigma$ is a homeomorphism.
\end{proof}

\begin{remark}\label{frankjones} (a) Given that (\ref{derivative1}) converges, it seems inevitable
that the pointwise derivative
$$\Sigma'(\theta)=\prod_{n=1}^{\infty} \left(\frac{1 - r_{n'}^2}{r_{n'}^2 + 2r_{n'}\cos(n'\Sigma_{n-1}(\theta) + q_{n'}) + 1}\right)$$
However we have not proven this. The general issue is the following. Suppose that $s_n$ and $s$ are nondecreasing functions and $s_n\to s$ at points of continuity of $s$. Does this imply
that $s_n' \to s'$ a.e. $[Leb]$? The answer is no. For example $s_n$ can be a sequence of staircase step functions converging to $s(\theta)=\theta$ (in which case $s_n'=0$ a.e., and $s'=1$). Suppose that $s_n\to s$ and $s_n'\to w$ a.e. Is there an inequality, $w\le s'$? This is unknown to us (see chapter 16 of \cite{Jones}, especially section C, for a venerable positive result).

(b) Even if the above derivative formula does hold, this does not imply that $d\Sigma$ is in the Lebesgue class (we additionally need to show $\int_0^{2\pi}\Sigma' d\theta=2\pi$). This is a zero-one question, and we do not know what to expect.
\end{remark}

\begin{lemma}\label{lemma10} Suppose $0\le \rho<1$ and let $X=-\log(1+2\rho \cos(\theta)+\rho^2)$, where $\theta$ is a uniformly distributed angle. Then $E(X)=0$ and
$$E(X^2)=2\sum_{k=1}^{\infty}\frac{1}{k^2}\rho^{2k}=2dilog(1-\rho^2)  $$
\end{lemma}

\begin{proof}If $\rho=0$, then clearly $E(X)=E(X^2)=0$.
$$\frac{d}{d\rho}E(X)=\frac{1}{2\pi}\int_{S^1}\frac{2(\rho+\cos(\theta))}{1+2\rho \cos(\theta)+\rho^2}d\theta=0$$
Thus $E(X)=0$ for all $\rho$. In a similar way
$$\frac{d}{d\rho}E(X^2)=\frac{4}{\rho}\log(\frac{1}{1-\rho^2})=4\sum_{k=1}^{\infty}\frac1{k^2}\rho^{2k-1}$$
This implies the formula for the second moment.
\end{proof}

\begin{theorem}\label{zeroderivative} If $r \notin l^2$ and the phases are i.i.d. and uniform, then almost surely $\Sigma_n'(\theta)\to 0$ a.e. $[Leb]$.
\end{theorem}

Of course we would like to believe this implies that $d\Sigma$ is almost surely singular with respect to Lebesgue measure, but this is uncertain (see Remark \ref{frankjones} above).

\begin{proof}We will use the same notation as in the proof of the preceding theorem. For each $\theta$
\begin{equation}\log\Sigma'_N(\theta) = \sum_{n=1}^N \left(\log(1-r_n^2) - \log(r_n^2 + 2 r_n \cos (v_n) + 1)\right)			 \end{equation}
If $r_n$ does not converge to zero, then this clearly diverges to $-\infty$. So we can suppose $r_n\to 0$.

Let $X_n =-\log(r_n^2 + 2 r_n \cos (v_n) + 1)$. The $X_n$ are independent random variables. Since $v_n$ is uniform, Lemma \ref{lemma10} implies that $E(X_n)=0$ and the variance $var(X_n)=2r_n^2+o(r_n^4)$ as $n\to\infty$. Let $S_N = X_1 + ... + X_N$. By the law of iterated logarithms,  there is a constant $c$ such that
almost surely $S_N\le c \sqrt{a_N\log\log a_N}$, where $a_N=\sum_{n\le N}r_n^2$. Hence $\log\Sigma'_N(\theta)$ is almost surely asymptotically bounded by
$-a_N + c\sqrt{a_N\log\log a_N}$. Thus almost surely $\log\Sigma'_N(\theta)$ diverges to $-\infty$, i.e. $\Sigma'_N(\theta)$ goes to $0$. By Fubini's theorem, almost surely $\Sigma'_N\to 0 $ a.e. $[Leb]$.

\end{proof}

\subsection{Onset of Cantor-like Behavior}\label{subsectCantor-like}

In this subsection, we consider the ordering $p(n)=n$.

The hypothesis in the following theorem is difficult. However using the example following the proof, we will clarify why the result is important.

\begin{theorem}\label{Cantor-like} Assume $r \notin l^2$, $\limsup_n r_n < 1$, and the phases $q_n$ are i.i.d. and uniform. Let $s_n = \sum_{k=1}^n r_k^2$. If $\sum n r_n \exp(-2 s_n + 2\sqrt{2\pi s_n\log\log s_n})$ is convergent, then almost surely, for almost all $\theta$ $[Leb]$ there exists a $\delta_\theta > 0$ such that $\Sigma(\theta + \delta_\theta) = \Sigma(\theta)$.
\end{theorem}

Fix $\theta$. Define the following:
	\begin{eqnarray}
	D_n(x) &=& \Sigma_n(\theta + x) - \Sigma_n(\theta)\\
	d_n(x) &=& x - \frac{2}{n}\arctan\left(\frac{2r_n\sin(n x /2) \cos(\beta_n(x)) + r_n^2\sin(nx)}{1 + 2r_n\cos(n x /2) \cos(\beta_n(x)) + r_n^2\cos(nx)}\right)\\
	U_n &=& \frac{1-r_n^2}{1 + 2 r_n \cos(\beta_n(0)) + r_n^2}\\
	p_n &=& \prod_{k=1}^n U_n
	\end{eqnarray}
where
$$\beta_n(x) = q_n + n\Sigma_{n-1}(\theta) + \frac{nx}{2}$$

\begin{remark}The domain of $d_n$ is a priori the set of $x$ for which the denominator is nonvanishing.
If $r_n<2^{1/2}-1$, then there is no constraint on $x$. To see this consider the roots $r(a,b)=(-b\pm\sqrt{b^2-a})/a$
of $1+2br+ar^2=0$ in the region $|a|,|b|\le 1$, $b^2-a\ge 0$. The mininum magnitude for these roots occurs in the
corners $a=-1, b=\pm 1$. In the following lemma we will use an analytic continuation to enlarge the domain of $d_n$.
\end{remark}

\begin{lemma}\label{lemma11} Fix $\theta$ as above.

(a) For fixed $x$ the $\beta_n(x)$ are i.i.d. and uniform random angles.

(b) $D_n(x)=d_n(D_{n-1}(x))$ for sufficiently small $x$; $d_n$ can be analytically continued so that
$d_n\in \widetilde{Homeo(S^1)}$ $D_n=d_n\circ D_{n-1}$ holds for all $x$.

(c)  $d_n(0)=0$ and $d_n'(0)=U_n$.

(d) There is a constant $B$, independent of $n,x$, such that
$d_n(x)\le U_n x(1+Bnr_nx)$ for all $n$ and $x\ge 0$.

\end{lemma}

\begin{proof} (a) is clear.

(b) For small $a$ and $b$
\begin{equation}\label{trigidentity}\arctan(a)-\arctan(b)=\arctan(\frac{a-b}{1+ab})\end{equation}
In general this equality holds modulo $\mathbb Z\pi$.

$$D_n(x)=\Phi_n(w_n,\Sigma_{n-1}(\theta+x))-\Phi_n(w_n,\Sigma_{n-1}(\theta))$$
\begin{equation}\label{eqn1}=D_{n-1}(x)-\frac{2}{n}\left(\arctan(\frac{r_n\sin(n\Sigma_{n-1}(\theta+x)+q_n)}{1+r_ncos(n\Sigma_{n-1}(\theta+x)+q_n)})
-\arctan(\frac{r_n\sin(n\Sigma_{n-1}(\theta)+q_n)}{1+r_ncos(n\Sigma_{n-1}(\theta)+q_n)})\right)\end{equation}
To simplify notation, let $T=n\Sigma_{n-1}(\theta+x)+q_n$ and $t=n\Sigma_{n-1}(\theta)+q_n$. Note $T-t=nD_{n-1}(x)$.
Since $a$ and $b$ are small, the identity (\ref{trigidentity}) implies that (\ref{eqn1}) equals
$$D_{n-1}(x)-\frac{2}{n}\arctan(\frac{\frac{r_n\sin(T)}{1+r_n\cos(T)}-\frac{r_n\sin(t)}{1+r_n\cos(t)} }
{1+\frac{r_n\sin(T)}{1+r_n\cos(T)}\frac{r_n\sin(t)}{1+r_n\cos(t)} } )$$
$$=D_{n-1}(x)-\frac{2}{n}\arctan(r_n\frac{\sin(T)-\sin(t)+\sin(T-t)r_n}
{1+(\cos(T)+\cos(t))r+\cos(T-t)r_n^2 })$$
$$=D_{n-1}(x)-\frac{2}{n}\arctan(r_n\frac{\sin(nD_{n-1}(x)+t)-\sin(t)+\sin(nD_{n-1}(x))r_n}
{1+(\cos(nD_{n-1}(x)+t)+\cos(t))r_n+\cos(nD_{n-1}(x))r_n^2 })
$$
Now observe that
$$\sin(nD_{n-1}(x)+t)-\sin(t)=\sin(n x' /2) \cos(\beta_n(x'))|_{x'=D_{n-1}(x)}$$
and
$$\cos(nD_{n-1}(x)+t)+\cos(t)=\cos(n x' /2) \cos(\beta_n(x'))|_{x'=D_{n-1}(x)}$$
This implies part (b) for small $x$. It follows that $d_n(x)=D_n^{-1}\circ D_{n-1}(x)$ for small $x$.
We can use this to analytically continue $d_n$. This implies (b).

(c) $d_n(0)=0$ follows immediately
from the definition of $d_n$. The derivative $d_n'(x)$ is given by a complicated formula.
But the evaluation at $x=0$ is given by the simple formula in (c).

(d) The assumption $\limsup_n (r_n) < 1$ implies that $U_n$ has positive lower bound independent of $n$. So the essential claim is that
there is a uniform bound $d_n''(x)<Bnr_n$, for all $n$ and $x\ge 0$. Note that when we differentiate $d_n$, the branching
issue in (b) vanishes. The explicit expression for $d_n''(x)$ is long (we used Maple). It is of the form $d_n''(x)=nr_n R $,
where $R$ is a rational function in $r_n$ and cosines and sines with arguments $nx$, $\frac{nx}{2}$, and $n\Sigma_{n-1}(\theta)+q_n+\frac{nx}{2}$. The expression for $R$, as a function of $r_n$ and these cosines and sines, does not depend on $n$. Using the fact that the cosines and sines are bounded by one, one can obtain a bound for $R$ which does not depend on $n,x$ (e.g.
$R=2nsin(n\Sigma_{n-1}(\theta)+q_n)r_n+o(r_n)$ as $r_n\to 0$). This proves (d).

\end{proof}

\begin{proof} (of Theorem \ref{Cantor-like}) Fix $\theta$. We will first show that almost surely, for $x$ small enough, $D_n(x) \rightarrow 0$ as $n\to\infty$.

As a random variable $U_n$ is identical to $-X_n$ in the proof of Theorem \ref{zeroderivative} above. Using the same argument in
that proof (using the Law of Iterated Logarithms), it follows that $p_n$ almost surely goes to $0$ at least as fast as $\exp({-2s_n + 2\sqrt{2\pi s_n \log\log s_n}})$.

Now define a sequence $c_n$ recursively by $c_n = U_n c_{n-1} (1 + n r_n B U_n c_{n-1})$. Notice that $D_n(c_0) \le c_n$ by (d) above. For sufficiently small $c_0$, we will show $c_n \rightarrow 0$ almost surely. Let $\lambda_n = \prod_{k=1}^n (1 + B n r_n p_{n})$. The assumption that $n r_n \exp(-2 s_n + 2\sqrt{2\pi s_n\log\log s_n})$ is summable implies that $\lambda_n$ has a limit $L$. Notice that for $\epsilon < \min_{\{n\in\mathbb{N}\}}(U_n/{\lambda_{n-1}})$, we will have that $c_{n-1} < \epsilon\lambda_{n-1}p_{n-1}$ implies $c_n < \epsilon \lambda_n p_n$ for all $n$, because in this case
$$ c_n = U_n c_{n-1} (1+ n r_n BU_n c_{n-1}) < \epsilon p_n \lambda_{n-1}  (1+B n r_n \epsilon \lambda_{n-1} p_{n-1}) < \epsilon p_n \lambda_{n-1}(1 + B n r_n p_n) = \epsilon \lambda_n p_n	$$
Hence in this case, we must have that $c_n \rightarrow 0$, because we know $\lambda_n$ converges while $p_n \rightarrow 0$. Hence we conclude that for $\delta < \min_{\{n\in\mathbb{N}\}}(U_n/{\lambda_{n-1}})$, we will have that $D_n(\delta) \le c_n \rightarrow 0$, so $D_n(\delta) = 0$ and hence $\Sigma(\theta+\delta) = \Sigma(\theta)$.

Since this applies to almost all $\theta$, Fubini's theorem now implies the statement in Theorem \ref{Cantor-like}.

\end{proof}

\begin{example} To understand the condition in Theorem \ref{Cantor-like}, consider $r_n = \sqrt{c} n^p$, for $p\in[-\frac12,0]$. We must check the summability of
$$	n r_n \exp(-2 s_n + 2\sqrt{2\pi s_n\log\log s_n})$$
In this case, if $p > - \frac12$ $s_n = \sum_{k=1}^n c k^{2p}$, which is between  $\frac{c}{2p+1} ((n+1)^{2p+1} - 1)$ and $\frac{c}{2p+1} n^{2p + 1}$. Thus
	\begin{eqnarray}
	&&n r_n \exp(-2 s_n + 2\sqrt{2\pi s_n\log\log s_n})\nonumber\\
	 &\le& \sqrt{c} n^{1+p}\exp\left(-\frac{2c (n+1)^{2p+1}}{2p+1} + \frac{2c}{2p+1} + 2\sqrt{ \frac{2\pi c n^{2p+1}}{2p+1} \log\log \left(\frac{c n^{2p+1}}{2p+1}\right) }\right)\nonumber
	\end{eqnarray}
This in fact goes to $0$ quite rapidly, due to the $\exp(-K n^{2p+1})$ term, and therefore is summable. For the case of $p = \frac12$, one has that $s_n = \sum r_n^2 \sim c\log{n} + c\gamma$, where $\gamma$ is the Euler-Mascheroni constant. In this case, we have:
	\begin{eqnarray}
	&&n r_n \exp(-2 s_n + 2\sqrt{2\pi s_n\log\log s_n})\nonumber\\
	&\sim& \sqrt{c} n^{\frac12}\exp\left(-{2c}\log{n} - 2c\gamma+  2\sqrt{ 2\pi c(\log n + \gamma) \log\log\log n }\right)\nonumber\\
	&=& \sqrt{c}e^{-2\gamma} n^{\frac12 - 2c} \exp\left(2\sqrt{ 2\pi c(\log n + \gamma) \log\log\log n }\right)\nonumber
	\end{eqnarray}
which, because of how slowly $\exp(\sqrt{\log n})$ diverges, will converge as long as $\frac12 - 2c < -1$, or equivalently $c > \frac34$. Hence, $r_n > \frac{\sqrt{3}}{2\sqrt{n}}$ is sufficient for $\Sigma(w)$ to almost surely exhibit a Cantor-like behavior. In particular, $r_n =\frac{1}{\sqrt{n}}$ is bad enough.
\end{example}

Theorem \ref{dalthorp1} (which we proved at the beginning of this section) states that $r\in l^2$ and i.i.d. uniform random angles implies almost sure invertibility. Thus this result pins down the transition between invertibility to non-invertibility to a relatively narrow window, because $r_n =\frac{1}{\sqrt{n}}$ is very nearly $l^2$. Our aim in the next Section is to pin this down even further.

\section{Comparison with Verblunsky Coefficients}\label{Verblunsky}

In this section we briefly compare root subgroup coordinates with Verblunsky coefficients (which are actually related to root subgroup
coordinates for the loop group $LSU(1,1)$).
We assume that the ordering is $p(n)=n$, because it does not seem natural to orthogonalize polynomials in
a nonstandard order. The goal is to explain the following diagram, where in the top (bottom) row we are
considering root subgroup coordinates (Verblunsky coefficients, respectively):

\begin{equation}\label{diagram}\begin{matrix} R.S.F.& &  \prod_{n=1}^{\infty}\Delta & \leftarrow &  \mathbf w^0\cap\prod_{n=1}^{\infty}\Delta&\leftarrow &   \mathbf w^{1/2} \cap\prod_{n=1}^{\infty}\Delta \\
 & & \downarrow a.s. & &\downarrow a.s. & &\updownarrow (?) \\
& & Prob(S^1)& \leftarrow& S^1\backslash Homeo(S^1)  &\leftarrow & S^1\backslash W^{1+1/2}Homeo \\
 & & \updownarrow  & &\uparrow a.s(?) & &\updownarrow  \\
V. & & \overline{\prod_{n=1}^{\infty}\Delta} & \leftarrow &  \mathbf w^0\cap\prod_{n=1}^{\infty}\Delta&\leftarrow &   \mathbf w^{1/2}\cap\prod_{n=1}^{\infty}\Delta \end{matrix}\end{equation}

For root subgroup coordinates (by Theorem \ref{increasingfn}) there is an almost sure map
$$\begin{matrix}
\prod_{n=1}^{\infty}\Delta  & \stackrel{a.s.}{\rightarrow} &CDF\\
  & & \downarrow \\ & & Prob(S^1) \end{matrix}$$

$$\begin{matrix} w & \stackrel{a.s.}{\rightarrow} & \Sigma(w,\theta)=\theta-2\sum_{n=1}^{\infty}\frac1n \Theta(1+w_n\sigma_{n-1}(\theta)^n)\\
   & & \downarrow \\
   & & \frac1{2\pi}d\Sigma =\text{weak}^*-\lim_{N\to\infty}\left(\prod_{n=1}^N
\frac{1-\vert w_n\vert^2}{\vert 1 + w_n\sigma_{n-1}(z)^n\vert^2}\right)\frac{d\theta}{2\pi}\end{matrix}$$
where the phases of the $w_n$ are i.i.d. and uniform (for $w$ such that $\sum \frac1n|w_n|<\infty$,
the map is surely defined).
The corresponding Verblunsky map (see \cite{Simon}) is given by
\begin{equation}\label{Verblunskymap}\prod_{n=1}^{\infty}\Delta \to Prob(S^1):\alpha \to
\mu(\alpha)= \text{weak}^*-\lim_{N\to\infty}
\frac{\prod_{n=1}^{N}(1-\vert \alpha_n\vert^2)}{\vert p_N(z)\vert^2}\frac{d\theta}{2\pi}\end{equation}
where $p_0=1$ and for $n>0$ the $p_n$ are defined by the Szego recursion relation
\begin{equation}\label{Szego} p_{n}(z) = zp_{n-1}(z)-\alpha_n^*z^{n-1}p_
{n-1}^{*}(z)\end{equation}
where $p^*(z)=p(1/z^*)^*$.

If $\alpha_n=0$ for $n>N$, then $p_{N+k}(z)=z^kp_N(z)$, and hence
$$\mu(\alpha)=
\frac{\prod_{n=1}^{N}(1-\vert \alpha_n\vert^2)}{\vert p_N(z)\vert^2}\frac{d\theta}{2\pi}$$
If $\alpha_n=0$ for $n\ne N$, then $p_{N}(z)=z^N-\alpha_N^*$, and hence
\begin{equation}\label{Verblunskymap2}\mu(\alpha)=\frac1{2\pi}d\Sigma(w)\end{equation}
where $w_n=0$ for $n\ne N$ and $w_N=-\alpha_N$.

The inverse of the Verblunsky map (for simplicity, restricted to measures with infinite support) is given by
$$Prob'(S^1) \to
\prod_{n=0}^{\infty}\Delta: \mu \to (\alpha_n)$$
where if $p_0 =1$, $p_1(z)$, $p_2(z)$, .. are the monic orthogonal polynomials corresponding
to the nontrivial measure $\mu$, then $\alpha_n =-p_{n+1}(0)^*$ (It is not evident that $|\alpha_n|<1$; this follows
by an argument using the recursion relation (\ref{Szego}) for the $p_n$; see Theorem 1.5.2 of \cite{Simon}).
We do not know how to calculate the composition of maps
$$\prod_{n=1}^{\infty}\Delta \to Prob(S^1) \to \prod_{m=1}^{\infty}\Delta: w \to \mu_w \to \alpha$$
for sequences with multiple nonzero terms. It is even quite complicated to calculate that
$$\alpha_1^*(w_1,w_2,0,..)=-\frac{w_1^* + w_1 w_2^*}{1 + w_1^2w_2^*}$$

Suppose that $\Sigma$ is a cdf which corresponds to $\mu(\alpha)$.
A famous theorem of Szego (see Theorem 8.1 of \cite{Simon}) asserts that
$$\prod_{n=1}^{\infty}(1-|\alpha_n|^2)=\exp(\frac{1}{2\pi}\int \log(\Sigma')d\theta)$$
Consequently $\alpha\in l^2$ if and only if $\int log(\Sigma')d\theta>-\infty$. In particular $\alpha\in l^2$ implies
that $\Sigma'>0$ a.e. $[Leb]$ and hence $\Sigma$ is strictly increasing. But $\alpha\in l^2$ does not imply that
$\Sigma$ is continuous (see Section 2.7 of \cite{Simon}, especially Example 2.7.5). This should be compared with the situation for root subgroup coordinates: in this case
$w\in l^2$ easily implies $\sum_{n>0}\frac1n |w_n|<\infty$, and hence $\Sigma(w)$ is continuous, but it can easily happen
that $\Sigma$ is not strictly increasing. For root subgroup coordinates we also know that if the phases of the $w_n$ are i.i.d.
and uniform, then $w\in l^2$ does imply that $\sigma(w)$ is a homeomorphism. It is apparently not known whether the corresponding statement is true for Verblunsky coefficients. This explains the middle vertical arrows in (\ref{diagram}).

Suppose that
$$\mu(\alpha)=e^f \frac{d\theta}{2\pi}+\mu_s, \qquad f(z)=\sum_{n=-\infty}^{\infty}f_nz^n$$
where $\mu_s$ is perpendicular to the Lebesgue class. Another famous theorem of Szego (with a refinement due to Ibragimov)
asserts that if $\mu_s=0$, then
\begin{equation}\label{Szego2}\prod_{n=1}^{\infty}(1-\vert
\alpha_{n}\vert^2)^{n}=\exp(-\sum_{k=1}^{\infty}k\vert f_k\vert^2)\end{equation}
(see Chapter 6 of \cite{Simon}). This implies a positive answer to the Verblunsky analogue of Question \ref{sobolevquestion} (which explains the right most vertical arrows in (\ref{diagram}).

\begin{corollary}\label{Szego3} $\alpha\in \mathbf w^{1/2}$ if and only if a cdf for $\mu(\alpha)$ belongs
to $W^{1+1/2}Homeo(S^1)$.
\end{corollary}

Szego's theorem (\ref{Szego2}), and a further consideration involving volume, suggests that (an appropriate probabilistic extension of) the map $M:f\to \alpha$ might induce an equality of the two probability measures
\begin{equation}\label{maintheorem}M_*\left(\prod_{k=1}^{\infty}\frac{\beta k}{\pi}e^{-\beta k\vert
f_k\vert^2}d\lambda(f_k)\right)\stackrel{?}{=}\prod_{n=1}^{\infty}
\frac{nb}{\pi}(1-\vert\alpha_{n}\vert^2)^{nb-1}d\lambda(\alpha_{n}) \end{equation}
where $\beta>0$ and $b=1+\beta$. The former measure is essentially so called multiplicative chaos). The striking feature is the shift in exponent, which arises from the following volume calculation:
for finite $N$, on a certain domain $D_N$,
\begin{equation}\label{volume}\chi_{D_N}\prod_{n=1}^N d\lambda(f_n)= \prod_{n=1}^N(1-|\alpha_n|^2)^{n-1} d\lambda(\alpha_n) \end{equation}
(This will appear elsewhere). If (\ref{maintheorem}) is correct, then for the measures on the right hand side of (\ref{maintheorem}), the Verblunsky analogue of the measures in Subsection \ref{applications}, there are two potential phase transitions, $b=1$, and a conjectural transition at $b=1+\sqrt{2}$ ($\beta=\sqrt{2}$) identified in \cite{AJKS}.

\section{Appendix: Smoothness Conditions for Homeomorphisms of $S^1$}\label{AppendixB}

For a map $\sigma: S^1 \to S^1$ which is 1-1 and onto, the inverse is also 1-1
and onto. For such a map, if $\sigma$ is continuous, then the inverse is also continuous.
However given a more general smoothness condition S for self-maps of $S^1$, the set of
homeomorphisms of $S^1$ satisfying condition S may or may not form a subgroup. We are interested in filtering
homeomorphisms in terms of groups, and hence we will want to impose bi-conditions on homeomorphisms and their inverses (We will say that a homeomorphism is bi-S to mean that the homeomorphism and its inverse satisfy condition S).
In this appendix we recall important examples. The most interesting examples
are of groups which arise because they fix some kind of geometric structure.

\subsection{Holder Type Conditions}

(1) For $s = 0$ and for $s \ge 1$, $C^sHomeo(S^1)$ is the topological group of orientation preserving
homeomorphisms of $S^1$ which, together with their inverses, are $C^s$.

(2) For fixed $0 < s < 1$, $C^s$ is not closed with respect to composition, and hence $C^sHomeo(S^1)$ fails to be a group (e.g.
$f(x)=x^3$ is a bi-$C^{1/3}$ homeomorphism of $\mathbb R$ and $f\circ f(x)=x^9$ is not a bi-$C^{1/3}$ homeomorphism). Thus it is problematic, in the group theoretic category, to filter homeomorphisms by Holder smoothness in the range $0<s<1$. The set of bi-Holder continuous homeomorphisms, $C^{0+}Homeo(S^1)$, does form a group.

(3) $\phi\in Homeo(S^1)$ is quasisymmetric if there is a constant $M$ such that
$$\frac1M\le |\frac{\phi(e^{i(\theta+t)})-\phi(e^{i\theta})}{\phi(e^{i(\theta)})-\phi(e^{i(\theta-t)})}|\le M$$
for all $\theta,t$ (see e.g. chapter 16 of \cite{GL}). There are other characterizations: $\phi$ is quasisymmetric if and only if it can be extended to a homeomorphism of the disk which is quasiconformal, if and only if it stabilizes the critical Sobolev class
$ W^{1/2}(S^1)$ (To properly formulate this, it is essential to note that a $W^{1/2}$ equivalence class of measurable functions
has a preferred representative which is defined off of a set of logarithmic capacity zero, and quasisymmetric maps preserve capacity; see \cite{NS}). The inverse of a quasisymmetric homeomorphism is also quasisymmetric. The set $QS(S^1)$ of quasisymmetric homeomorphisms of $S^1$ is a group, and it is also naturally a Banach manifold, but it is not a topological group. Any quasisymmetric homeomorphism is Holder continuous of order $s$, where $s = 1/K$ and the homeomorphism has a $K$-quasiconformal extension to $\Delta$.

The conditions for homeomorphisms which we have considered are summarized as: for $s > 1$
\begin{equation}\label{inclusions}bi-C^s \subset bi-C^1 \subset QS \subset bi-Holder\subset C^0\end{equation}
There is a sharp transition at $s = 1$.

\subsection{Sobolev Type Conditions}

(4) For $S=W^{1,L^1}$, the set $AC(S^1)$ of bi-absolutely continuous homeomorphisms is
a group (The inverse of an absolutely continuous homeomorphism is not necessarily absolutely continuous, hence the condition on the inverse is essential).
$AC(S^1)$ is the group of homeomorphisms which fix the Lebesgue class $[d\theta]$ of $S^1$.
As a consequence this group acts unitarily on half-densities (of the Lebesgue class) on the circle with finite norm.

(5) For $s > 3/2$, $W^{s}Homeo(S^1)$, the set of orientation preserving homeomorphisms of $S^1$
which, together with their inverses, are $W^s := W^{s,L^2}$ (smooth of order s in the $ L^2$ Sobolev sense), is a topological
group. More generally, for a compact $d$-manifold $X$, the set of homeomorphisms
of $X$ which are smooth of order s in the $L^2$ Sobolev sense is a topological group,
provided $s > 1 + d/2$; see \cite{Marsden}.

(6) For $s = 3/2$, the critical $L^2$ Sobolev case, a subtle adjustment in the
definition is apparently required (We do not know how to show by example that this is essential, but this seems certain).
Define
$$W^{1+1/2}Homeo:= \{\sigma\in AC(S^1) : \ln(\Sigma') \in W^{1/2,L^2}\}$$
This is a topological group.

To summarize the group conditions that are most important for us, analogous to (\ref{inclusions}),
there are inclusions
\begin{equation}\label{inclusions}W^{1+s,L^2} \rightarrow W^{1+1/2,L^2}  \rightarrow QS  \rightarrow C^{0+}\rightarrow Homeo(S^1)\end{equation} where $s>1/2$. For $L^2$-Sobolev exponents there is a sharp transition at $s =1/2$, analogous to
the transition for Holder exponents at $s = 1$.

\end{document}